\documentclass[11pt]{amsart}
\usepackage{amssymb,mathrsfs,graphicx,enumerate,color}

\usepackage{mathtools}
\usepackage{hyperref} 
\usepackage{comment}
\newcommand{\RR}{{\mathbb R}}

\newcommand{\Bcal}{\mathcal{B}}

\newcommand{\Dcal}{\mathcal{D}}
\newcommand{\Ecal}{\mathcal{E}}
\newcommand{\Fcal}{\mathcal{F}}
\newcommand{\Gcal}{\mathcal{G}}
\newcommand{\Hcal}{\mathcal{H}}
\newcommand{\Ical}{\mathcal{I}}
\newcommand{\Jcal}{\mathcal{J}}

\newcommand{\Mcal}{\mathcal{M}}

\newcommand{\Rcal}{\mathcal{R}}

\newcommand{\Xcal}{\mathcal{X}}

\newcommand{\VBar}{\overline{V}}

\newcommand{\Ubar}{\bar{U}}

\renewcommand{\hbar}{\bar{h}}

\newcommand{\ubar}{\bar{u}}
\newcommand{\vbar}{\bar{v}}

\newcommand{\Util}{\tilde{U}}

\newcommand{\htil}{\tilde{h}}

\newcommand{\util}{\tilde{u}}
\newcommand{\vtil}{\tilde{v}}

\newcommand{\abs}[1]{\left|#1\right|}

\newcommand{\norm}[1]{\left\|#1\right\|}

\DeclareMathOperator{\supp}{supp}

\renewcommand{\a}{\alpha}
\renewcommand{\b}{\beta}
\newcommand{\g}{\gamma}
\renewcommand{\d}{\delta}
\newcommand{\e}{\varepsilon}

\newcommand{\z}{\zeta}
\newcommand{\et}{\eta}
\renewcommand{\th}{\theta}

\renewcommand{\l}{\lambda}

\newcommand{\m}{\mu}
\newcommand{\n}{\nu}
\newcommand{\x}{\xi}

\renewcommand{\r}{\rho}
\newcommand{\s}{\sigma}
\renewcommand{\t}{\tau}

\newcommand{\ps}{\psi}

\renewcommand{\O}{\Omega}

\newcommand{\rd}{\partial}

\newcommand{\rhotil}{\tilde{\rho}}

\newcommand{\rbar}{\bar{\rho}}
\newcommand{\kbar}{\overline{\kappa}}
\newcommand{\kubar}{\underline{\kappa}}
\newcommand{\bbar}{\overline{\beta}}
\newcommand{\bubar}{\underline{\beta}}

\newcommand{\one}[1]{\mathbf{1}_{\{#1\}}}

\usepackage{colortbl}
\definecolor{black}{rgb}{0.0, 0.0, 0.0}
\definecolor{red}{rgb}{1.0, 0.5, 0.5}

\title[   ]{Stability of Riemann Shocks for isothermal Euler by Inviscid limits of global-in-time large Navier-Stokes flows}

\author[Eo]{Saehoon Eo}
\address[Saehoon Eo]
{ Department of Mathematics, \newline
Stanford University \\
CA 94305, USA}
\email{eosehoon@stanford.edu}

\author[Eun]{Namhyun Eun}
\address[Namhyun Eun]
{ Department of Mathematical Sciences, \newline
Korea Advanced Institute of
Science and Technology \\
Daejeon 34141, Korea}
\email{namhyuneun@kaist.ac.kr}
 
\author[Kang]{Moon-Jin Kang}
\address[Moon-Jin Kang]
{ Department of Mathematical Sciences, \newline
Korea Advanced Institute of
Science and Technology \\
Daejeon 34141, Korea}
\email{moonjinkang@kaist.ac.kr}

\author[Oh]{HyeonSeop Oh}
\address[HyeonSeop Oh]
{ Department of Mathematical Sciences, \newline
Korea Advanced Institute of
Science and Technology \\
Daejeon 34141, Korea}
\email{ohs2509@kaist.ac.kr}

\newtheorem{theorem}{Theorem}[section]
\newtheorem{lemma}{Lemma}[section]

\newtheorem{proposition}{Proposition}[section]
\newtheorem{remark}{Remark}[section]

\newcommand{\deo}{\d_0}

\newcommand{\red}[1]{\textcolor{red}{#1}}

\numberwithin{figure}{section}
\newcommand{\beq}{\begin{equation}}
\newcommand{\eeq}{\end{equation}}
\newcommand{\bsp}{\begin{split}}
\newcommand{\esp}{\end{split}}

\newcommand{\wc}{\rightharpoonup}

\newtheorem{theo}{Theorem}[section]

\newtheorem{lem}[theo]{Lemma}

\newcommand{\vt}{{\tilde{v}}}
\newcommand{\ut}{{\tilde{u}}}

\newcommand{\Ut}{{\tilde{U}}}

\newcommand{\vtn}{{\tilde{v}^\nu}}
\newcommand{\utn}{{\tilde{u}^\nu}}
\newcommand{\htn}{{\tilde{h}^\nu}}

\newcommand{\vn}{{v^\nu}}
\newcommand{\un}{{u^\nu}}
\newcommand{\hn}{{h^\nu}}

\newcommand{\Phif}[2]{\Phi\Big(\frac{#1}{#2}\Big)}
\newcommand{\bpf}[1]{\noindent\textbf{Proof of \eqref{#1}:}} 
 
\newcommand{\step}[1]{\vskip0.2cm \noindent{\it Step #1:} }

\begin{document}

\date{\today}


\thanks{\textbf{Acknowledgment.} This work was supported by Samsung Science and Technology Foundation under Project Number SSTF-BA2102-01.}

\begin{abstract}
In this paper, we study the isothermal gas dynamics.
We first establish the global existence of strong solutions to the one-dimensional isothermal Navier-Stokes system for smooth initial data without any smallness conditions, assuming that the initial density has strictly positive lower bound.
The existence result allows for possibly degenerate viscosity coefficients and admits different asymptotic states at the far fields.
We then prove a contraction property for the strong solutions perturbed from viscous shocks, yielding uniform estimates with respect to the viscosity coefficients.
This covers any large perturbations, and consequently, we establish the inviscid limits and their stability estimate.
In other words, we demonstrate the stability of Riemann shocks to the one-dimensional isothermal Euler system in the class of vanishing viscosity limits of the associated Navier-Stokes system.
\end{abstract}
\maketitle \centerline{\date}

\tableofcontents


\section{Introduction}
\setcounter{equation}{0}

In this article, we consider the one-dimensional compressible isothermal Navier-Stokes system.
In the Lagrangian mass coordinates, the system is given as
\begin{align}
\left\{
\begin{aligned} \label{inveq}
    &v_t^\nu-u_x^\nu = 0, \\
    &u_t^\nu+p(v^\nu)_x = \nu\left(\m(v)\frac{u_x^\nu}{v^\nu}\right)_x,
\end{aligned}
\right.
\end{align}
where \(v,u,p(v)\) and \(\m(v)\) respectively represent the specific volume, fluid velocity, pressure law and the viscosity coefficient.
We consider the polytropic gas, especially an isothermal case, and the possibly degenerate viscosity coefficient near vacuum :
\begin{equation} \label{pressure}
p(v)=1/v, \qquad \m(v)= b v^{-\a}
\end{equation}
for \(b>0\) and \(\a\in[0,1]\), but here we set \(b=1\) for simplicity as the value of \(b\) does not affect our analysis. 
Moreover, \(\n>0\) is a parameter for vanishing viscosity limits.

Our study encompasses both the scenario of commonly assumed case of constant viscosity, i.e., \(\a=0\), and the situation where the viscosity degenerates in the vicinity of a vacuum.
From a physical perspective, the viscosity is temperature-dependent (see Chapman-Cowling \cite{ChapmanCowling90}) and, in isothermal conditions, it varies with density, or equivalently with the specific volume \(v\) when described in Lagrangian mass coordinates.

The one-dimensional compressible Navier-Stokes system with constant viscosity coefficient, have been extensively investigated by many authors.
For smooth initial data with density uniformly bounded away from vacuum, the existence of weak solutions to the viscous and heat-conducting flows was first established by Kazhikov and Shelukhin \cite{Kazhikov77}.
In contrast, for discontinuous initial data under the same assumption that the initial density remains strictly positive, Serre \cite{Serre86} addressed both the isothermal and barotropic cases.
Extending this framework, Hoff \cite{Hoff98} demonstrated the global existence of weak solutions even for large discontinuous data that may approach different states as \(x\to\pm\infty\).
Furthermore, he also showed that the solution maintains strictly positive density, thereby excluding the formation of vacuum states in finite time.
In higher dimension cases, analogous result for smooth data in the context of viscous and heat-conducting flows was obtained in \cite{MN79} by Matsumura-Nishida, and later Feireisl \cite{Feireisl04} extended the theory by establishing the global existence results even when the initial density is allowed to vanish, while Hoff \cite{Hoff95} investigated both the isothermal and barotropic cases for discontinuous data.
On the other hands, for the density-dependent viscosity cases, although Liu-Xin-Yang \cite{LXY98} derived the compressible Navier-Stokes equations with degenerate viscosity, the mathematical theory for the isothermal case with degenerate viscosity remains largely undeveloped.

Even though the isothermal flow more accurately captures certain physical phenomena than barotropic model-for example, gas flow through a long tube (see \cite{Shapiro})-it has predominantly been treated as a special case of barotropic flows.
While numerous studies include the isothermal case, many focus exclusively on the regime where the adiabatic constant \(\g>1\).
As a result, in comparison to the barotropic regime, the isothermal Navier-Stokes system remains relatively less understood.
In particular, Mellet-Vasseur \cite{MVSIMA} established the existence of global large strong solutions for the one-dimensional isentropic Navier-Stokes equations under assumptions of the degenerate viscosity, different asymptotic states at the far fields, and the initial density bounded away from vacuum. The result was extended by Haspot \cite{Haspot} to the case of $\alpha\in (1/2,1]$. Then, Constantin-Drivas-Nguyen-Pasqualotto \cite{CDNP} covered the case of $\alpha\ge 0$ and $\gamma\in [\alpha, \alpha+1]$ with $\gamma>1$, but they studied the system on the periodic domain. The result \cite{CDNP} is extended by Kang-Vasseur \cite{KV-Existence} to the whole space with different asymptotic states.
Such results, however, are not yet available for the isothermal case, which we address in the present paper.\\

Our interest also lies in the inviscid limits of the strong solutions of isothermal Navier-Stokes systems to construct a perturbation class on which the uniqueness and stability of Riemann shock of the compressible Euler systems would hold. 
Note that at least formally, as \(\n \to 0+\), the system \eqref{inveq} converges to the isothermal Euler system
\begin{align}
\left\{
\begin{aligned} \label{Euler}
    &v_t-u_x = 0, \\
    &u_t+p(v)_x = 0.
\end{aligned}
\right.
\end{align}


The well-posedness of the compressible Euler systems has been a fundamental and long-standing subject of intensive study by many researchers.
In a celebrated work, Glimm \cite{Glimm65} established the global-in-time existence of small BV entropic solutions.
Several years after Glimm's pioneering work, Nishida \cite{Nishida68} employed Glimm's random choice method to analyze the case of isothermal flow and established the global existence of solutions for arbitrary BV data without imposing any smallness condition.
Moreover, in the context of the uniqueness and stability, Dafermos \cite{D96} and Diperna \cite{Diperna79} proved that Lipschitz solutions are unique and \(L^2\)-stable within the class of bounded entropic solutions.

On the other hands, addressing discontinuous patterns presents an entirely different challenge.
The previous uniqueness and stability results for special discontinuous solutions, including Riemann shocks, required suitable regularity conditions, such as locally BV or strong trace property (see Chen-Frid-Li \cite{ChenFridLi02} and Vasseur et al. \cite{ChenKrupaVasseur22, GKV23, KV16, Vasseur16}).
On the case of isothermal gas, the stability of arbitrary BV solutions under strong trace property was recently proved in \cite{ChengIsothermal} by Cheng.
However, studying the uniqueness and stability of shocks without any technical conditions for perturbations remains a challenging open problem.
A natural approach to avoid such technical conditions is to consider a set of inviscid limits.
The concept of inviscid limits was first introduced by Stokes \cite{stokes1998difficulty}, and has since been employed to construct entropy solutions to the isentropic Euler system (see \cite{DiPerna83, DiPerna83CMP, Hoff-Liu, Goodman-Xin, Yu99}).
Notably, Chen-Perepelitsa \cite{ChenPerepelitsa10} established the  convergence of solutions to the Navier-Stokes system with constant viscosity coefficients (i.e., \(\a=0\)) towards an entropy solution to the isentropic Euler system, utilizing the method of compensated compactness.
It is noteworthy that, to date, there are no available results employing compensated compactness directly for the isothermal Navier-Stokes system.
However, there exists a result addressing the case of asymptotically isothermal pressure \cite{SMSS20}.
We also refer to \cite{Huang-Zhen02} for the application of the compensated compactness method in the isothermal case; however, their analysis is based on an artificial viscous system rather than the physical Naiver-Stokes systems.

As a breakthrough result in the context of stability to the discontinuous patterns, Bianchini-Bressan \cite{BianchiniBressan05} constructed a global unique and stable entropy solution with small BV initial data, obtained as the vanishing artificial viscosity limit of full parabolic systems.
In a more physically relevant settings, Kang-Vasseur \cite{KV-Inven} first proved the stability of weak Riemann shocks to the isentropic Euler system within a class of inviscid limits from the associated barotropic Navier-Stokes system.
This result was extended to the cases of composite waves involving shocks (see \cite{HWWW24,KV-JDE}).
Recently, Chen-Kang-Vasseur \cite{ChenKangVasseur24arxiv} constructed global unique and stable small BV solutions to the isentropic Euler system from Navier-Stokes systems.
We also refer to \cite{EEK-BNSF} for the result on the stability of Riemann shocks for the full Euler system within a class of inviscid limits from the Brenner-Navier-Stokes-Fourier systems.

However, to the best of our knowledge, the well-posedness of Riemann shocks through inviscid limits for the isothermal Euler system remains open.
In this paper, we present the first result on this problem.

\subsection{Global existence of strong solutions to Navier-Stokes systems}
Our first result is on the global existence of large strong solutions to the isothermal Navier-Stokes system.
In contrast to the inviscid limit part, the solution existence argument is based on the Eulerian coordinates, 
describes the isothermal Navier-Stokes system in the following form:
\begin{align}
\left\{
\begin{aligned} \label{NS-Eulerian}
    &\r_t + (\r u)_x = 0, \\
    &(\r u)_t + (\r u^2 + p(\r))_x = (\m(\r)u_x)_x,
\end{aligned}
\right.
\end{align}
where \(\r,u,p(\r)\) and \(\m(\r)\) respectively represent the density, fluid velocity, pressure law, and viscosity coefficient.
The pressure law and the viscosity coefficient are given by
\begin{equation} \label{pressure-Eulerian}
p(\r)=\r \qquad \m(\r)= b \r^\a,
\end{equation}
for \(b>0\).
We remark that a slight abuse of notations (on the space variable \(x\), velocity \(u\), pressure \(p\), and the viscosity \(\m\)) is introduced here, but it should not cause any confusion.

As mentioned earlier, for the barotropic cases, Mellet-Vasseur \cite{MVSIMA} established the global existence of a large strong solution for the degenerate power \(\a \in [0,1/2)\), and afterwards, it was extended by Kang-Vasseur \cite{KV-Existence} to the case where \(\a\ge1/2\).
Thus, for the isothermal case, we here follow \cite{MVSIMA} and \cite{KV-Existence} to establish the global existence of the case when \(\a\in[0,1/2)\) first and then we proceed to cope with the case when \(\a\in[1/2,1]\).

We first present the global existence (and uniqueness) result of a large strong solution to \eqref{NS-Eulerian} for the case when \(\m(\r)\) degenerates strictly slower than \(\r^{1/2}\) near \(\r=0\).

\begin{theorem} \label{thm:MV}
Let \(\r_0\) and \(u_0\) be the initial data such that 
\[
0<\kubar_0\le\r_0(x)\le\kbar_0, \quad \r_0-\rbar\in H^2(\RR), \quad u_0-\ubar \in H^1(\RR),
\]
for some constants \(\kubar_0,\kbar_0\), and where \(\rbar\) and \(\ubar\) are smooth monotone functions such that
\begin{equation} \label{sm-end-E}
\rbar(x)=\r_\pm \quad \text{and} \quad
\ubar(x)=u_\pm, \quad \text{when } \pm x \ge 1.
\end{equation}
Let \(\n:\RR^+\to\RR^+\) be a function such that for some constants \(C>0\) and \(q\in[0,1/2)\).
\begin{equation} \label{MV-lb}
\n(y) \ge
\left\{
\begin{aligned}
&Cy^q &&\forall y\le1, \\
&C &&\forall y\ge1,
\end{aligned}\right.
\end{equation}
and
\begin{equation} \label{MV-ub}
\n(y) \le  C+Cy \quad \forall y\ge0.
\end{equation}
Then, there exists a global-in-time unique strong solution \((\r,u)\) of \eqref{NS-Eulerian}-\eqref{pressure-Eulerian} with \(\m=\n\) such that for any \(T>0\), there exist positive constants \(\bubar(T)\) and \(\bbar(T)\) satisfying
\begin{align*}
&\r-\rbar \in L^\infty(0,T;H^1(\RR)),\\
&u-\ubar \in L^\infty(0,T;H^1(\RR)) \cap L^2(0,T;H^2(\RR)),\\
&\bubar(T)\le\r(t,x)\le\bbar(T), \quad \forall(t,x)\in[0,T]\times\RR.
\end{align*}
\end{theorem}

We then present the existence results for the case when \(\m(\r)=b\r^\a\) with \(\a\in[1/2,1]\).
This case requires more regular initial data with a technical condition \eqref{scaling} as in \cite{CDNP,KV-Existence}.

\begin{theorem} \label{thm:KV-AP}
Assume \(\a\in[0,1]\).
Let \(\r_0\) and \(u_0\) be the initial data such that
\begin{align*}
&\r_0-\rbar \in H^k(\RR), \quad u_0-\ubar \in H^k(\RR), &&\text{for some integer } k\ge4, \\
&0<\kubar_0 \le \r_0(x) \le \kbar_0, \quad \forall x \in \RR, &&\text{for some constants } \kubar_0,\kbar_0,
\end{align*}
and
\begin{equation} \label{scaling}
\rd_x u_0(x) \le \r_0(x)^{\g-\a}=\r_0(x)^{1-\a}, \quad \forall x \in \RR,
\end{equation}
where \(\rbar\) and \(\ubar\) are the smooth monotone functions satisfying \eqref{sm-end-E}.
Then, there exists a global-in-time unique smooth solution \((\r,u)\) to \eqref{NS-Eulerian}-\eqref{pressure-Eulerian} such that for any \(T>0\),
\begin{align*}
&\r-\rbar \in L^\infty(0,T;H^k(\RR)), \\
&u-\ubar \in L^\infty(0,T;H^k(\RR)) \cap L^2(0,T;H^{k+1}(\RR)).
\end{align*}
Moreover, there exist constants \(\kubar(T)\) and \(\kbar(T)\) such that 
\[
0<\kubar(T) \le \r(t,x) \le \kbar(T), \quad \forall (t,x)\in[0,T]\times\RR.
\]
\end{theorem}

\begin{remark}
Note that the system \eqref{NS-Eulerian} is equivalent to \eqref{inveq} with \(\n=1\).
Thus, these two results above provide a class of global-in-time solutions smooth enough, in which a certain contraction property holds.
More precisely, these results define the following function space in which the solutions reside:
\begin{multline*}
\Xcal_T \coloneqq
\{(v,u) \mid 
v-\underline{v}, u-\underline{u} \in L^\infty(0,T;H^1(\RR)), u-\underline{u} \in L^2(0,T;H^2(\RR)),\\ 
v^{-1}, v \in L^\infty((0,T)\times\RR) \}
\end{multline*}
where \(\underline{v}\) and \(\underline{u}\) are smooth monotone functions such that
\begin{align} \label{sm-end}
\underline{v}(x)=v_{\pm}, \quad \underline{u}(x)=u_{\pm}, \quad \text{for } \pm x \ge 1.
\end{align}
\end{remark}

\subsection{Uniqueness and stability of Riemann shocks to Euler system}
To study the stability of Riemann shocks under the vanishing viscosity limit framework, it is natural to employ the viscous shocks, which is the viscous counterpart of Riemann shocks, of the associated Navier-Stokes systems.
It is known that the system \eqref{inveq} admits a viscous shock wave connecting two end states \((v_-,u_-)\) and \((v_+,u_+)\), provided that the two end states satisfy both the Rankine-Hugoniot condition and the Lax entropy condition (See \cite{MatsumuraWang10}):
\begin{align}
\begin{aligned} \label{end-con}
&\exists~\s \quad \text{s.t.}~\left\{
\begin{aligned}
&-\s(v_+-v_-) -(u_+-u_-) =0, \\
&-\s(u_+-u_-) +p(v_+)-p(v_-) =0,
\end{aligned} \right. \\
&\text{and either \(v_->v_+\) and \(u_->u_+\) or \(v_-<v_+\) and \(u_->u_+\) holds.}
\end{aligned}
\end{align}
In other words, for given constant states \((v_-,u_-)\) and \((v_+,u_+)\) satisfying \eqref{end-con}, there exists a viscous shock wave (\(\vtn, \utn\))(\(\x\)) with \(\x = x - \s t\) as a solution of
\begin{align}
\left\{
\begin{aligned} \label{shock_0}
    &-\s (\vtn)'-(\utn)' = 0, \quad ' = \frac{d}{d\x}, \\
    &-\s (\utn)'+p(\vtn)' = \nu\left(\m \frac{(\utn)'}{\vtn}\right)'.
\end{aligned}
\right.
\end{align}
If \(v_->v_+\), (\(\vtn, \utn\))(\(x-\s t\)) is a \(1\)-shock wave with \(\s=- \sqrt{-\frac{p_+-p_-}{v_+-v_-}}\), where \(p_\pm \coloneqq p(v_\pm)\).
On the contrary, if \(v_-<v_+\), it is a \(2\)-shock wave with \(\s= \sqrt{-\frac{p_+-p_-}{v_+-v_-}}\).

Let (\(\vbar,\ubar\)) be an associated entropic (inviscid) shock wave (or Riemann shock) connecting the two end states \((v_-,u_-)\) and \((v_+,u_+)\) satisfying \eqref{end-con} as follows:
\begin{equation} \label{shock-0}
(\vbar,\ubar)(x-\s t)= \left\{
\begin{aligned}
(v_-,u_-) \quad \text{if } x-\s t<0, \\
(v_+,u_+) \quad \text{if } x-\s t>0.
\end{aligned} \right.
\end{equation}

To study stability of shocks, we use the relative entropy associated to the entropy of \eqref{Euler} as follows: for any positive functions \(v_1,v_2\) and any functions \(u_1, u_2\), the relative entropy is given by
\begin{equation} \label{eta_def}
\et((v_1,u_1)|(v_2,u_2)) \coloneqq \Phif{v_1}{v_2}+\frac{1}{2}(u_1-u_2)^2
\end{equation}
where \(\Phi(z) \coloneqq z-1-\log z\).
Note that \(\Phi(z_1/z_2)\) is the relative functional associated with the strictly convex function \(Q(z)\coloneqq-\log z\).
Namely,
\[
Q(z_1|z_2)
\coloneqq Q(z_1)-Q(z_2)-Q'(z_2)(z_1-z_2)
= \Phif{z_1}{z_2}.
\]
Unlike the isentropic case, the isothermal case permits the consideration of physical entropy, which closely resembles the form of entropy for the full Euler system.
Naturally, the relative entropy also possesses a similar structure.
We refer to \cite{EEK-BNSF} for the relative entropy of the full Euler system.

We will consider limits of solutions to the Navier-Stokes system for the first component \(z_1\), i.e., \(v_1\). Since we only obtain uniform bounds in \(L^1\) for the solutions, the limits could be regarded as measures on \(\RR^+\times\RR\). This could be physical, as a possible appearance of vacuum.
Hence, we need to extend the notion of relative entropy to measures defined on \(\RR^+ \times \RR\).
Since we handle perturbations of a Riemann shock, it is enough to extend the definition only for the case when we compare a measure \(dv\) with a simple function \(\vbar\) which takes only two values \(v_-\) and \(v_+\).
Let \(v_a\) be the Radon-Nikodym derivative of \(dv\) with respect to the Lebesgue measure and \(dv_s\) be its singular part, i.e., \(dv=v_a dtdx+dv_s\).
Then, we define the relative functional as 
\begin{align*}
d\Phif{v}{\vbar} \coloneqq \Phif{v_a}{\vbar}dtdx + \frac{1}{\VBar(t, x)} dv_s(t,x),
\end{align*}
where \(\VBar\) is given by
\[
\VBar(t,x)=
\begin{cases}
\max(v_-,v_+) & \text{for } (t,x)\in \overline{\O_M} (\eqqcolon \text{the closure of }\O_M), \\
\min(v_-,v_+) & \text{for } (t,x)\notin \overline{\O_M},
\end{cases}
\]
and \(\O_M=\{(t,x)|\vbar(t,x)=\max(v_-,v_+)\}\).
We use \(\VBar\), not \(\vbar\), because it has to be defined at every point to deal with \(dv_s\).
Notice that in this case, the relative entropy is a measure itself.
Furthermore, if \(v\in L^{\infty}(\RR^+;L^\infty(\RR)+\Mcal(\RR))\) and its Radon-Nikodym derivative \(v_a\) is away from \(0\), then \(d\Phi(v/\vbar)\) does belong to \(L^{\infty}(\RR^+;\Mcal(\RR))\)
, where \(\Mcal\) denotes the space of nonnegative Radon measures.

The following theorem is on stability and uniqueness of the entropy shocks to \eqref{Euler}:
\begin{theorem}\label{thm_inviscid}
Let \(\a\in[0,1]\) be any constant.
For each \(\n>0\), consider the system \eqref{inveq}-\eqref{pressure}.
For a given constant state \((v_-, u_-)\in \RR^+\times \RR\), there exists a constant \(\e_0>0\) such that for any \(\e<\e_0\) and any \((v_+, u_+)\in \RR^+\times \RR\) satisfying \eqref{end-con} with \(\abs{p(v_+)-p(v_-)} = \e\), the following statement holds. \\
Let \((\vtn, \utn)\) be a viscous shock connecting the two end states \((v_-, u_-)\) and \((v_+, u_+)\) as a solution of \eqref{shock_0}.
Then for a given initial datum \(U^0=(v^0, u^0)\) of \eqref{Euler} satisfying 
\[
\Ecal_0 \coloneqq \int_\RR \et((v^0, u^0)|(\vbar, \ubar))dx < \infty,
\]
the followings hold. \\
(i) (Well-prepared initial data) There exists a sequence of smooth functions \(\{(v_0^\nu, u_0^\nu)\}_{\nu>0}\) on \(\RR\) such that
\begin{equation} \label{ini_conv}
\begin{aligned}
&\lim_{\n\to 0}v_0^\nu = v^0, \quad \lim_{\n\to 0}u_0^\nu = u^0 \quad a.e., \quad v_0^\nu > 0, \\
&\lim_{\n\to 0}\int_\RR \Phif{v_0^\nu}{\vtn} + \frac{1}{2}\Big(u_0^\nu-\n \frac{(v_0^\nu)_x}{(v_0^\nu)^{\a+1}}- \utn+\n\frac{(\vtn)_x}{(\vtn)^{\a+1}}\Big)^2 dx = \Ecal_0.
\end{aligned}
\end{equation}
(ii) For any given \(T>0\), let \(\{(\vn, \un)\}_{\n>0}\) be a sequence of solutions in \(\Xcal_T\) to \eqref{inveq} with the initial datum \((v_0^\nu, u_0^\nu)\) as above. 
Then there exist limits \(v_\infty\) and \(u_\infty\) such that as \(\n\to 0\) (up to a subsequence),
\begin{equation}\label{wconv}
\vn\wc v_\infty, \quad \un\wc u_\infty ~\text{ in }~ \Mcal_{loc}((0, T)\times \RR),
\end{equation}
where \(v_\infty\) lies in \(L^\infty(0, T, L^\infty(\RR)+\Mcal(\RR))\) and \(\Mcal_{loc}((0, T)\times \RR)\) is the space of locally bounded Radon measures. \\
In addition, there exist a shift \(X_\infty\in BV((0, T))\) and a constant \(C>0\) such that \(d\Phi(v_\infty/\vbar)\in L^\infty(0, T;\Mcal(\RR))\), and for almost every \(t\in (0, T)\), 
\begin{equation} \label{uni-est}
\int_{x\in\RR} d\Phi(v_\infty/\vbar(x-X_\infty(\cdot)))(t) 
+\int_{\RR} \frac{\abs{u_\infty(t,x) - \ubar(x-X_\infty(t))}^2}{2}dx \le C\Ecal_0.
\end{equation}
Moreover, the shift \(X_\infty\) satisfies 
\begin{equation}\label{X-control}
\abs{X_\infty(t)-\s t} \le \frac{C(T)}{\abs{v_--v_+}}(\sqrt{\Ecal_0}+\Ecal_0).
\end{equation}
Therefore, entropy shocks \eqref{shock-0} with small amplitudes of the isothermal Euler system \eqref{Euler} are stable (up to shifts) and unique in the class of inviscid limits of solutions to the Navier-Stokes system \eqref{inveq}. 
\end{theorem}

\begin{remark}
    (1) Theorem \ref{thm_inviscid} provides the uniqueness and stability of weak Riemann shocks in the class of inviscid limits of solutions to the isothermal Navier-Stokes systems. For the uniqueness, if \(\Ecal_0 = 0\), then \eqref{X-control} implies that \(X_\infty(t) = \s t\) for a.e. \(t \in (0,T)\). Consequently, from \eqref{uni-est}, we obtain
    \[
        \int_{x\in\RR} d\Phi(v_\infty(t,x)/\vbar(x-\s t))dx 
        +\int_{\RR} \frac{\abs{u_\infty(t,x) - \ubar(x-\s t)}^2}{2}dx =0,
    \]
    where \(d v_\infty = v_a dt dx + dv_s\), and the singular part \(v_s\) vanishes. Thus, we have
    \[
    v_\infty(t,x) = \vbar(x - \s t), \quad u_\infty(t,x) = \ubar(x - \s t), \quad \text{a.e. } (t,x) \in [0,T] \times \RR.
    \]
    (2) The smallness of shock amplitude does not play any role in the inviscid limit process. The restriction on its size arises from Theorem \ref{thm_main}.
\end{remark}


The cornerstone in the proof of Theorem \ref{thm_inviscid} is a certain uniform-in-\(\n\) estimate for any large perturbations of viscous shocks to \eqref{inveq}.
It can be obtained from a contraction property of any large perturbations of viscous shocks to \eqref{inveq} with a fixed \(\n=1\):

\begin{align}
\left\{
\begin{aligned} \label{main0}
    &v_t-u_x = 0, \\
    &u_t+p(v)_x = \left(\m(v) \frac{u_x}{v}\right)_x.
\end{aligned} \right.
\end{align}
As in \cite{KV21,KV-Inven}, we first introduce the following relative functional \(E(\cdot|\cdot)\) to measure the contraction: for any positive functions \(v_1,v_2,\) and any functions \(u_1,u_2,\) 
\begin{equation} \label{E-rel-ent}
E((v_1,u_1)|(v_2,u_2)) \coloneqq \Phif{v_1}{v_2}
+ \frac{1}{2} \Big(u_1-\frac{(v_1)_x}{(v_1)^{\a+1}} -u_2+\frac{(v_2)_x}{(v_2)^{\a+1}} \Big)^2.
\end{equation}
The functional \(E\) is associated to the BD entropy (see Bresch-Desjardins \cite{BD02, BD03, BD06, BDL03}).
Since \(\Phi(\cdot/\cdot)\) is positive definite, so is \eqref{E-rel-ent}, i.e., it holds that \(E((v_1,u_1))|((v_1,u_2)) \ge 0\), and 
\[
E((v_1,u_1)|(v_1,u_2)) = 0 ~\text{ a.e.} \quad \Leftrightarrow \quad  (v_1,u_1)=(v_2,u_2) ~\text{ a.e.}
\]

The following theorem provides a contraction property which is measured by the relative functional \eqref{E-rel-ent}.
\begin{theorem}\label{thm_main}
Let \(\a\in[0,1]\) be any constant.
Consider the system \eqref{main0} with \eqref{pressure}. 
For a given constant state \((v_-,u_-)\in\RR^+\times\RR\), 
there exist constants \(\e_0, \d_0>0\) such that the following statement holds.\\
For any \(\e<\e_0\), \(\d_0^{-1}\e<\l<\d_0\), and any \((v_+,u_+)\in\RR^+\times\RR\) satisfying \eqref{end-con} with \(\abs{p(v_+)-p(v_-)} = \e\), there exists a smooth monotone function \(a\colon \RR\to \RR^+\) with \(\lim_{x\to\pm\infty} a(x) = 1+a_{\pm}\) for some constants \(a_-\) and \(a_+\) with \(\abs{a_+-a_-}=\l\) such that the followings hold. \\
Let \(\Ut\coloneqq (\vt,\ut)\) be the viscous shock connecting \((v_-,u_-)\) and \((v_+,u_+)\) as a solution of \eqref{shock_0} with \(\n=1\).
For a given \(T>0\), let \(U\coloneqq (v,u)\) be a solution in \(\Xcal_T\) to \eqref{main0} with a initial datum \(U_0\coloneqq (v_0,u_0)\) satisfying \(\int_\RR E(U_0|\Ut) dx<\infty\). 
Then there exist a shift \(X\in W^{1,1}(0,T)\) and a constant \(C>0\) (independent on \(\e,\l\) and \(T\)) such that 
\begin{align*}
&\int_\RR a(x) E\big(U(t,x+X(t))| \Ut(x)\big) dx \\
&\qquad +\d_0\frac{\e}{\l} \int_0^T \int_\RR  \abs{\s a'(x)} \Phi\left(v(s,x+X(s))/\vt(x)\right) dx ds \\
&\qquad +\d_0 \int_0^T \int_\RR a(x)
v^{1-\a}(s, x+X(s))\abs{\rd_x\Big(p(v(s, x+X(s)))-p(\vt(x))\Big)}^2 dx ds \\
&\le \int_\RR a(x) E\big(U_0(x)|\Ut(x)\big) dx,
\end{align*}
and 
\begin{align}
\begin{aligned} \label{est-shift}
&\abs{\dot{X}(t)}\le \frac{C}{\e^2}\left(f(t)+1 \right), \quad \text{for a.e. } t\in[0,T], \\
&\text{for some positive function \(f\) satisfying} \quad \norm{f}_{L^1(0,T)} \le\frac{\l}{\d_0\e}\int_\RR E(U_0|\Ut) dx.
\end{aligned}
\end{align}
\end{theorem}

\section{Proof of Theorem \ref{thm:MV} and \ref{thm:KV-AP}: Existence} \label{sec:existence}
\setcounter{equation}{0}
In this section, we prove the existence results by following the arguments for the barotropic cases which can be found in \cite{MVSIMA} and \cite{KV-Existence}.
Note that in \cite{MVSIMA}, the proof consists of four components: an a priori estimate based on the relative entropy;  an a priori estimate based on the Bresch-Desjardins entropy, which is motivated by \cite{BD02, BD03, BD06, BDL03}; upper and lower bounds for the density; standard arguments for parabolic equations.
Considering that the structure of the relative entropy differs from that in the barotropic case, we provide a detailed proof of the first component.
However, the second component heavily depends on the form of the Navier-Stokes system and thus \eqref{NS-Eulerian} can be regarded as a special case of the barotropic case with \(\g=1\), and so we omit the details.
For the third, Mellet-Vasseur \cite{MVSIMA} used Sobolev embedding on \(\r^{\frac{\g-1}{2}}\) to obtain the upper bound, but in our case, this quantity reduces to \(1\).
Thus, we use a logarithmic function rather than \(\r^{\frac{\g-1}{2}}\), as discussed below.
For the last one, there are no significant differences from the barotropic case, so we omit it.

Now we introduce the relative entropy associated with \eqref{NS-Eulerian} and \eqref{pressure-Eulerian}.
The entropy pair \((\et,G)\) (which again is a notation abuse but with no confusion) is given by
\[
\et(U) = \frac{1}{2} \r u^2 + \r \log \r, \qquad
G(U) = \frac{1}{2} \r u^3 + \r u \log \r + \r u,
\]
and thus the relative entropy is defined for any solutions \(U\) and \(\Util\) by
\begin{align*}
\et(U|\Util) &\coloneqq \et(U) -\et(\Util) -\nabla \et(\Util)(U-\Util) 
= \frac{1}{2}\r (u-\util)^2 + \r \Phif{\rhotil}{\r},
\end{align*}
where \(\Phi(z)= z-1-\log z\).
Recall \(\rbar\) and \(\ubar\) in \eqref{sm-end-E}, and we denote 
\(\Ubar = \begin{pmatrix} \rbar \\ \rbar \ubar \end{pmatrix}\).
It is easy to see that there exists a constant \(C_1>0\) (depending only on \(\inf \rbar\)) such that for every \(\r>0\) and for every \(x\in\RR\), we have
\begin{equation} \label{r=0}
\r \le C_1 (1+\r\Phif{\rbar}{\r}), \qquad 
\liminf_{\r\to0}\r\Phif{\rbar}{\r} \ge C_1^{-1}.
\end{equation}

Then, we present the following lemma, which provides an a priori estimate.
\begin{lemma} \label{lem:uni-est1}
Let \((\r,u)\) be a solution to \eqref{NS-Eulerian}-\eqref{pressure-Eulerian} with \(\m=\n\) satisfying \eqref{MV-ub}.
Assume that the solution \((\r,u)\) satisfies the entropy inequality
\begin{equation} \label{ent-ineq}
\rd_t \et(U) + \rd_x (G(U)-\n(\r) u u_x) + \n(\r) \abs{u_x}^2 \le 0.
\end{equation}
Assume that the initial data \((\r_0,u_0)\) satisfies
\begin{equation} \label{ini-finite}
\int_\RR \et(U_0|\Ubar) dx
= \int_\RR \bigg( \frac{1}{2}\r_0 (u_0-\ubar)^2 + \r_0 \Phif{\rbar}{\r_0} \bigg) dx < +\infty.
\end{equation}
Then, for every \(T>0\), there exists a constant \(C(T)>0\) such that
\begin{equation} \label{uni-est1}
\sup_{[0,T]} \int_\RR \bigg(\frac{1}{2}\r (u-\ubar)^2 + \r \Phif{\rbar}{\r}\bigg) dx
+\int_0^T \int_\RR \n(\r) \abs{u_x}^2 dx dt \le C(T).
\end{equation}
The constant \(C(T)\) depends only on \(T>0\), \(\Ubar\), the initial value \(U_0\), and the constant \(C>0\) appearing in \eqref{MV-ub}.
\end{lemma}
Thanks to the locally quadratic property of the relative entropy, \eqref{ini-finite} holds under the assumptions of Theorem \ref{thm:MV}.

\begin{proof}
First of all, we observe (by for instance \cite{Dafermos79-ARMA}) that
\begin{align*}
\rd_t \et(U|\Ubar)
&= \big(\rd_t \et(U) + \rd_x (G(U)-\n(\r)u u_x) \big)
- \rd_x \big(G(U)-\n(\r)u u_x \big) \\
&\qquad
-\nabla\et(\Ubar)(\rd_t U + \rd_x A(U))
+\nabla\et(\Ubar)\rd_x A(\Ubar)
+\nabla\et(\Ubar)\rd_x A(U|\Ubar) \\
&\qquad
+\rd_x \big(\nabla G(\Ubar)(U-\Ubar)\big)
-\nabla^2 \et(\Ubar) \rd_x A(\Ubar) (U-\Ubar).
\end{align*}
Here we used \(\rd_t \Ubar = 0\), and \(A(U)\) and \(A(U|\Ubar)\) are as follows:
\[
A(U)=\begin{pmatrix} 0 \\ \r u^2 + \r \end{pmatrix}, \quad
A(U|\Ubar) \coloneqq A(U) - A(\Ubar) - \nabla A(\Ubar)(U-\Ubar)
=\begin{pmatrix} 0 \\ \r(u-\ubar)^2 \end{pmatrix}.
\]
Since \(U\) is a solution to \eqref{NS-Eulerian}-\eqref{pressure-Eulerian} with the the entropy inequality \eqref{ent-ineq}, we get
\begin{align*}
\rd_t \et(U|\Ubar)
&\le -\n(\r)\abs{u_x}^2
- \rd_x \big(G(U)-\n(\r)u u_x \big)
- \ubar \rd_x\big(\n(\r) u_x \big)
+\nabla\et(\Ubar)\rd_x A(\Ubar) \\
&\qquad
+\nabla\et(\Ubar)\rd_x A(U|\Ubar)
+\rd_x \big(\nabla G(\Ubar)(U-\Ubar)\big)
-\nabla^2 \et(\Ubar) \rd_x A(\Ubar) (U-\Ubar). 
\end{align*}
We then integrate over \(x\in\RR\) and use \(\supp \rd_x \Ubar \subset [-1,1]\) to get
\begin{align*}
&\frac{d}{dt}\int_\RR \et(U|\Ubar) dx + \int_\RR \n(\r)\abs{u_x}^2 dx
\le -\big(G(U)-\n(\r)u u_x \big) \Big|_{-\infty}^{\infty}
+ \int_{-1}^{1} \ubar_x \n(\r) u_x dx \\
&\qquad
-\int_{-1}^1 \rd_x \nabla \et(\Ubar) \big(A(\Ubar) + A(U|\Ubar) \big) dx
-\int_{-1}^1 \nabla^2 \et(\Ubar) \rd_x A(\Ubar) (U-\Ubar) dx.
\end{align*}
We use Young's inequality to obtain
\[
\int_{-1}^{1} \ubar_x \n(\r) u_x dx
\le \frac{1}{2} \int_{-1}^1 \n(\r) \abs{u_x}^2 dx + \norm{\ubar_x}_{L^\infty(\RR)}^2 \int_{-1}^1 \n(\r) dx,
\]
and so it holds by \(\abs{A(U|\Ubar)} \le \et(U|\Ubar)\) that there is a constant \(C>0\) depending on \(\|\Ubar\|_{W^{1,\infty}}\)
\begin{equation} \label{for-est1}
\begin{aligned}
\frac{d}{dt}\int_\RR \et(U|\Ubar) dx + \frac{1}{2}\int_\RR \n(\r)\abs{u_x}^2 dx
&\le C\int_{-1}^1 \n(\r) dx + C\int_{-1}^1 \abs{U-\Ubar} dx \\
&\qquad + C\int_{-1}^1 \et(U|\Ubar) dx + C.
\end{aligned}
\end{equation}
It remains to show that the right-hand side can be controlled by \(\et(U|\Ubar)\).
To this end, we use \eqref{MV-ub} with \eqref{r=0} to find that 
\[
C\int_{-1}^1 \n(\r) dx \le C + C \int_{-1}^1 \r dx \le C + C \int_{-1}^1 \r \Phif{\rbar}{\r} dx
\le C + C \int_{-1}^1 \et(U|\Ubar) dx.
\]
Next, using \eqref{r=0}, we get
\begin{align*}
\int_{-1}^1 \abs{U-\Ubar} dx
&\le \int_{-1}^1 \abs{\r-\rbar} dx + \int_{-1}^1 \r \abs{u-\ubar} dx + \int_{-1}^1 \abs{\ubar(\r-\rbar)} dx \\
&\le C + C \int_{-1}^1 \r dx + \int_{-1}^1 \r(u-\ubar)^2 dx \\
&\le C + C \int_{-1}^1 \r \Phif{\rbar}{\r} dx + \int_{-1}^1 \r(u-\ubar)^2 dx
\le C + C \int_{-1}^1 \et(U|\Ubar) dx,
\end{align*}
which together with \eqref{for-est1} gives
\[
\frac{d}{dt} \int_\RR \et(U|\Ubar) dx + \frac{1}{2} \int_\RR \n(\r)\abs{u_x}^2 dx
\le C + C\int_{-1}^1 \et(U|\Ubar) dx,
\]
and Gronwall's lemma establishes Lemma \ref{lem:uni-est1}.
\end{proof}

Now, we derive an \(L^\infty\) bound for the density \(\r\).
For that, we need the a priori estimate from Bresch-Desjardins entropy, which is given by the following:
for any \(C^2\) function \(\n(\r)\) with \eqref{MV-ub}, and any solution \((\r,u)\) to \eqref{NS-Eulerian}-\eqref{pressure-Eulerian} such that 
\[
u-\ubar\in L^2(0,T;H^2(\RR)), \quad
\r-\rbar\in L^\infty(0,T;H^1(\RR)), \quad
0<m\le\r\le M,
\]
there exists \(C(T)>0\), which depends only on \(T>0\), \(\Ubar\), the initial value \(U_0\), and the constant \(C\) appearing in \eqref{MV-ub}, such that the following inequality holds:
\begin{equation} \label{uni-est2}
\sup_{[0,T]} \int_\RR \bigg(\frac{1}{2}\r \abs{(u-\ubar)+\rd_x(\varphi(\r))}^2 + \r \Phif{\rbar}{\r}\bigg) dx
+\int_0^T \int_\RR \rd_x(\varphi(\r))\rd_x\r dx dt \le C(T),
\end{equation}
with \(\varphi\) such that
\(
\varphi'(\r)=\frac{\n(\r)}{\r^2}.
\)
Note that \eqref{uni-est1} and \eqref{uni-est2} with \eqref{MV-lb} provide a strictly positive lower bound for the density as in \cite{MVSIMA}.
Furthermore, as mentioned above, \eqref{uni-est2} corresponds to \cite[Lemma 3.2]{MVSIMA} and can be derived in the same manner.

We then present the following proposition, which gives an \(L^\infty\) bound for the density.

\begin{proposition} \label{prop:r-ub}
For every \(T>0\), there exists a constant \(\kbar(T)\) such that 
\[
\r(t,x) \le \kbar(T) \quad \forall (t,x)\in(0,T)\times\RR.
\]
\end{proposition}

\begin{proof}
Let \(r(t,x)\coloneqq \sup(\r(t,x),1) = 1 + (\r(t,x)-1)_+\).
Notice that 
\[
\rd_x r = \rd_x \r \one{\r\ge1}.
\]
From \eqref{uni-est2} with \eqref{MV-lb}, it holds that \(\rd_x (\log r)\) is bounded in \(L^2((0,T)\times\RR)\).\\
Moreover, for every compact subset \(K\) of \(\RR\), we have
\begin{align*}
\int_K \abs{\rd_x (\log r)}dx
&=\int_K \abs{\frac{\rd_x r}{r}}dx
\le \bigg(\int_K \r dx\bigg)^{1/2}\bigg(\int_K \frac{1}{\r^3} (\rd_x \r)^2 dx\bigg)^{1/2}\\
&\le \bigg(\int_K \r dx\bigg)^{1/2}\bigg(\int_K \r (\varphi'(\r))^2 (\rd_x \r)^2 dx\bigg)^{1/2}
\end{align*}
and so using \eqref{r=0} we get 
\[
\int_K \abs{\rd_x (\log r)}dx
\le C\bigg(\abs{K}+\int_K \r\Phif{\rbar}{\r}dx\bigg)^{1/2}
\bigg(\int_K \r (\varphi'(\r))^2 (\rd_x \r)^2 dx\bigg)^{1/2}.
\]
Since 
\[
\log r = \log \r \one{\r\ge1} \le \r \le C_1 (1+ \r \Phif{\rbar}{\r}),
\]
we deduce that 
\[
(\log r) \text{ is bounded in } L^\infty(0,T;W_{loc}^{1,1}(\RR)),
\]
and the \(W^{1,1}(K)\) norm of \(\log r(t,\cdot)\) only depends on \(\abs{K}\).
Thus, Sobolev embedding yields an \(L^\infty((0,T)\times\RR)\) bound for \(\log r\), and this directly implies Proposition \ref{prop:r-ub} as we desired.
\end{proof}

Then, by following the standard parabolic arguments in \cite{MVSIMA} to the results obtained thus far, the proof of Theorem \ref{thm:MV} is completed.

The proof of Theorem \ref{thm:KV-AP} follows from the same argument as in the barotropic case \cite{KV-Existence} and is outlined as follows.
Firstly, we perturb the viscosity coefficient \(\m(\r)\) by adding \(\e\r^{1/4}\), where \(\e\) is a small parameter.
We then apply Theorem \ref{thm:MV} to obtain a global strong solution \((\r_\e,u_\e)\) to \eqref{NS-Eulerian}-\eqref{pressure-Eulerian} with the perturbed viscosity coefficient \(\m_\e\), their \(H^1\) regularity and a positive lower bound for \(\r_\e\).
Next, we introduce a new quantity \(w_\e \coloneqq -p(\r_\e)+\m_\e(\r_\e)\rd_x u_\e\) so-called the active potential, which is motivated by \cite{CDNP}, to eliminate the \(\e\)-dependence of the lower bound.
The remainder follows from standard analysis. Thus, we omit the proof.

\section{Proof of Theorem \ref{thm_main}: Contraction Property} \label{sec:theo}
\setcounter{equation}{0}
\subsection{Preliminaries: the method of $a$-contraction with shifts}
\subsubsection{Transformation of the system \eqref{main0}}
We here provide an equivalent version of Theorem \ref{thm_main} as in \cite{KV21,KV-Inven}.
First of all, we introduce a new effective velocity 
\[
h \coloneqq u - \frac{v_x}{v^{\a+1}}.
\]
Indeed, for \(\a=0\), \(h=u-(\log v)_x\) and for \(\a\in(0,1]\), \(h=u+\frac{1}{\a}(\frac{1}{v^\a})_x\).
Then, the system \eqref{main0} is transformed into
\begin{align}
\left\{
\begin{aligned} \label{main-0}
    &v_t-h_x = \left(\frac{v_x}{v^{\a+1}}\right)_x = - (v^\b p(v)_x)_x, \\
    &h_t+p(v)_x = 0,
\end{aligned} \right.
\end{align}
where \(\b \coloneqq 1-\a \in [0,1]\).
We consider viscous shock waves with the suitably small amplitude \(\e=\abs{p(v_+)-p(v_-)}\ll 1\), so let \(\s_\e\) denote a shock speed with the amplitude \(\e\), which is determined by \eqref{end-con}.
Then, setting \(\htil \coloneqq \util - \frac{\vtil'}{\vtil^{\a+1}}\), the shock wave \((\vtil,\htil)(x-\s_\e t)\) satisfies the following:
\begin{align}
\left\{
\begin{aligned} \label{shock}
    &-\s_\e \vtil'-\util' = -(\vtil^\b p(\vtil)')', \\
    &-\s_\e \htil'+p(\vtil)' = 0, \\
    &\lim_{\x\to\pm\infty}(\vtil,\htil)(\x) = (v_\pm,u_\pm).
\end{aligned}
\right.
\end{align}
For simplification of our analysis, we rewrite \eqref{main-0} into the following system, based on the change of variable \((t,x)\mapsto(t,\x=x-\s_\e t)\):
\begin{align}
\left\{
\begin{aligned} \label{main}
    &v_t - \s_\e v_\x - h_\x = -(v^\b p(v)_\x)_\x, \\
    &h_t - \s_\e h_\x +p(v)_\x = 0, \\
    &v|_{t=0}=v_0, \quad h|_{t=0}=u_0 - \frac{(v_0)_x}{(v_0)^{\a+1}} \eqqcolon h_0.
\end{aligned}
\right.
\end{align}
For the global-in-time existence of solutions to \eqref{main}, we consider the function space:
\[
\Hcal_T \coloneqq \{(v,h) \mid v-\underline{v} \in C(0,T;H^1(\RR)),\, h-\underline{u} \in C(0,T;L^2(\RR)),\,
v^{-1} \in L^\infty((0,T)\times\RR)\}
\]
where \(\underline{v}\) and \(\underline{u}\) are as in \eqref{sm-end}.

Theorem \ref{thm_main} is a direct consequence of the following theorem on the contraction of shocks to the system \eqref{main}.

\begin{theorem}\label{thm_main3}
Let \(\a\in[0,1]\) be any constant satisfying \(\a\in[0,1]\).
Consider the system \eqref{main} with \eqref{pressure}.
For a given constant state \((v_-,u_-)\in\RR^+\times\RR\), 
there exist constants \(\e_0, \d_0>0\) such that the following statement holds.\\
For any \(\e<\e_0\), \(\d_0^{-1}\e<\l<\d_0\), and any \((v_+,u_+)\in\RR^+\times\RR\) satisfying \eqref{end-con} with \(\abs{p(v_+)-p(v_-)} = \e\), there exists a smooth monotone function \(a\colon \RR\to \RR^+\) with \(\lim_{x\to\pm\infty} a(x)=1+a_{\pm}\) for some constants \(a_-\) and \(a_+\) with \(\abs{a_+-a_-}=\l\) such that the followings hold.\\
Let \(\Ut\coloneqq (\vtil,\htil)\) be the viscous shock connecting \((v_-,u_-)\) and \((v_+,u_+)\) as a solution of \eqref{shock}.
For a given \(T>0\), let \(U\coloneqq (v,h)\) be a solution in \(\Hcal_T\) to \eqref{main} with a initial datum \(U_0\coloneqq (v_0,h_0)\) satisfying \(\int_\RR \et(U_0|\Ut) dx<\infty\). 
Then there exist a shift \(X\in W^{1,1}(0,T)\) and a constant \(C>0\) (independent on \(\e,\l\) and \(T\)) such that 
\begin{align}
\begin{aligned}\label{cont_main3}
&\int_\RR a(\x) \et\big(U(t,\x+X(t))| \Ut(\x)\big) d\x \\
&\qquad +\d_0\frac{\e}{\l} \int_0^T \int_\RR \abs{\s_\e a'(\x)} \Phi\left(v(s,\x+X(s))/\vt(\x)\right) d\x ds \\
&\qquad +\d_0 \int_0^T \int_\RR a(\x)
v^\b(s,\x+X(s))\abs{\rd_\x\left(p(v(s, \x+X(s)))-p(\vt(\x))\right)}^2 d\x ds \\
&\le \int_\RR a(\x) \et\big(U_0(\x)|\Ut(\x)\big) d\x, \\
\end{aligned}
\end{align}
and 
\begin{align}
\begin{aligned} \label{est-shift3}
&\abs{\dot{X}(t)}\le \frac{C}{\e^2}\left(f(t)+1 \right), \quad \text{for a.e. } t\in[0,T], \\
&\text{for some positive function \(f\) satisfying} \quad \norm{f}_{L^1(0,T)} \le\frac{\l}{\d_0\e}\int_\RR \et(U_0|\Ut) d\x.
\end{aligned}
\end{align}
\end{theorem}

\begin{remark}
It is enough to prove Theorem \ref{thm_main3} for \(2\)-shocks.
For \(1\)-shocks, the change of variables \(x \to -x, u \to -u, \s_\e\to -\s_\e\) gives the corresponding result.
Thus, from now on, we consider a \(2\)-shock \((\vtil,\htil)\) only, that is, \(v_-<v_+, u_->u_+\) and 
\(
\s_\e = \sqrt{-\frac{p(v_+)-p(v_-)}{v_+-v_-}} > 0.
\)
\end{remark}

\begin{remark}\label{rem-vh}
Note that the solution \((v,u)\in\Xcal_T\) to \eqref{main0} gives a solution \((v,h)\) to \eqref{main} within the function space \(\Hcal_T\).
Indeed, since \(v_t = u_x \in L^2(0,T;H^1(\RR))\), we obtain \(v-\underbar{v} \in C(0,T;H^1(\RR))\). In order to show \(h - \underbar{u} \in C(0,T;L^2(\RR))\), we observe that for \((v,u) \in \Xcal_T\),
\[
h - \underbar{u} = u - \underbar{u} - \frac{v_x}{v^{\a+1}} \in L^\infty(0,T;L^2(\RR)).
\]
Moreover, applying Sobolev embedding, we obtain
\begin{align*}
    u_t &= -p'(v)v_x - (\a+1) \frac{v_x u_x}{v^{\a+2}} + \frac{u_{xx}}{v^{\a+1}} \in L^2(0,T;L^2(\RR)),\\
    \left(\frac{v_x}{v^{\a+1}}\right)_t &= -(\a+1)\frac{v_t v_x}{v^{\a+2}} + \frac{v_{xt}}{v^{\a+1}} = -(\a+1)\frac{u_x v_x}{v^{\a+2}} + \frac{u_{xx}}{v^{\a+1}} \in L^2(0,T;L^2(\RR)),
\end{align*}
which implies \(h_t \in L^2(0,T;L^2(\RR))\). Thus, we conclude that \(h - \underbar{u} \in C(0,T;L^2(\RR))\).
\end{remark}

\subsubsection{Global and local estimates on the relative quantities}
We here present useful inequalities on \(\Phi\) and \(p\) that will be crucially used for the proofs of main results.
First, the following lemma provides some global inequalities on the relative functional \(\Phi(\cdot/\cdot)\) corresponding to the convex function \(Q(v)=-\log v,\,v>0\).

\begin{lemma} \label{lem-pro}
For a given constant \(v_->0\), there exist constants \(c_1,c_2,c_3>0\) such that the following inequalities hold. \\
\begin{enumerate}
    \item For any \(w\in[\frac{1}{2}v_-,2v_-]\),
    \begin{equation} \label{rel_Phi}
    \begin{aligned}
    c_1\abs{v-w}^2 &\le \Phif{v}{w} \le c_1^{-1} \abs{v-w}^2 \quad 
    &&\text{for all } v \in \Big(\frac{1}{3}v_-,3v_-\Big), \\
    c_2 \abs{v-w} &\le \Phif{v}{w} \quad 
    &&\text{for all } v \in \RR^+\setminus\Big(\frac{1}{3}v_-,3v_-\Big).
    \end{aligned}
    \end{equation}
    \item Moreover, if \(0<w\le u \le v\) or \(0<v \le u \le w\), then
    \begin{equation} \label{Phi-sim}
    \Phif{v}{w} \ge \Phif{u}{w},
    \end{equation} 
    and for any \(\d_* \in (0,w/2)\), there exists a constant \(c_3>0\) such that if, in addition, \(w\in[\frac{1}{2}v_-,2v_-]\), \(\abs{w-v} > \d_*\), and \(\abs{w-u}\le\d_*\), then
    \begin{equation} \label{rel_Phi-L}
    \Phif{v}{w}-\Phif{u}{w} \ge c_3 \abs{v-u}.
    \end{equation}
\end{enumerate}
\end{lemma}
\begin{proof}
The proof can be found in \cite[Lemma 3.1]{EEK-BNSF}.
\end{proof}

We now provide some local estimates on \(p(v|w)\) and \(\Phi(v/w)\) for \(\abs{v-w}\ll 1\), which can be derived from Taylor expansion.

\begin{lemma} \label{lem:local}
For a given constant \(v_->0\), there exists a constant \(\d_*>0\) such that
for any \(\d \in (0,\d_*)\), the followings are true.
\begin{enumerate}
    \item For any \(v,w\in\RR^+\) satisfying \(\abs{1/v-1/w} \le \d\) and \(\abs{1/w-1/v_-} \le \d\),
    \begin{equation} \label{Phi-est1}
    \frac{1}{2}\Big(\frac{v}{w}-1\Big)^2 - \frac{1}{3} \Big(\frac{v}{w}-1\Big)^3 \le \Phif{v}{w} \le \frac{1}{2}\Big(\frac{v}{w}-1\Big)^2.
    \end{equation}
    \item For any \(v,w\in\RR^+\) satisfying \(\abs{p(v)-p(w)}<\d\) and \(\abs{p(w)-p(v_-)}<\d\),
    \begin{equation} \label{p-est1}
    p(v|w) \le \Big(\frac{1}{p(w)}+C\d\Big) \abs{p(v)-p(w)}^2
    \end{equation}
    where the relative pressure is defined as 
    \[
    p(v|w) =p(v)-p(w)-p'(w)(v-w).
    \]
    \item For any \(v,w\in\RR^+\) with \(\abs{p(w)-p(v_-)} \le \d\), if they satisfy either \(\Phi(v/w)<\d\) or \(\abs{p(v)-p(w)} < \d\), then
    \begin{equation} \label{v-quad}
    \abs{v-w}^2 \le C \Phif{v}{w},
    \end{equation}
    \begin{equation} \label{p-quad}
    \abs{p(v)-p(w)}^2 \le C \Phif{v}{w}.
    \end{equation}
\end{enumerate}
\end{lemma}
\begin{proof}
\bpf{Phi-est1} For the proof of \eqref{Phi-est1}, we refer to \cite[Lemma 3.2]{EEK-BNSF}.\\
\textbf{Proofs of \eqref{v-quad} and \eqref{p-quad}:} First, observe that \eqref{rel_Phi} implies \(\min\{c_1 \abs{v-w}^2, c_2 \abs{v-w}\} \leq \Phi(v/w)\).
If \(\Phi(v/w) \leq  \d < \d_* \ll 1\), then \(\abs{v-w} \ll 1\), which gives \(v_-/2 < v < 2v_-\) and \(c_1 \abs{v-w}^2 \leq \Phi(v/w)\). Therefore, we obtain 
\begin{equation}\label{quad-est}
    \abs{p(v) - p(w)}^2 \leq \abs{p'(v_-/2)}\abs{v-w}^2 \leq c_1^{-1}\abs{p'(v_-/2)}\Phif{v}{w}.
\end{equation}
If \(\abs{p(v) - p(w)} \leq \d\), then it follows from \eqref{Phi-est1} that 
\[
    \Phif{v}{w} \leq C\abs{p(v) - p(w)}^2 \leq \d,
\]
which gives \eqref{quad-est}.
Thus, both \eqref{v-quad} and \eqref{p-quad} hold.
\end{proof}

Throughout the remainder of this section and the succeeding section, \(C\) denotes positive constants which may change from line to line, but which stays independent on \(\e\) (the shock strength) and \(\l\) (the total variation of the function \(a\)).
We will consider two smallness conditions, one on \(\e\), and the other on \(\e/\l\).
In the argument, \(\e\) will be far smaller than \(\e/\l\).

\subsubsection{Properties of small shock waves}
In this subsection, we present useful properties of the \(2\)-shock wave \((\vtil,\htil)\) with small amplitude \(\e\).
In the sequel, without loss of generality, we consider the \(2\)-shock wave \((\vtil,\htil)\) satisfying \(\vtil(0)=\frac{v_-+v_+}{2}\).
Notice that the estimates in the following lemma also hold for \(\htil\) since we have \(\htil' = \frac{p'(\vtil)}{\s_\e}\vtil'\) and \(C^{-1} \le \frac{p'(\vtil)}{\s_\e} \le C\).
That said, since the estimates for \(\vtil\) are enough in our analysis, we state the estimates only for \(\vtil\).
\begin{lemma} \label{lem-VS}
For a given constant state \((v_-,u_-)\in\RR^+\times\RR\), there exist positive constants \(\e_0,C,C_1,C_2\) such that for any \(\e \in (0,\e_0)\), the followings are true.\\
Let \(\vtil\) be the \(2\)-shock wave with amplitude \(\abs{p(v_+)-p(v_-)}=\e\) and such that \(\vtil(0)=\frac{v_-+v_+}{2}\).
Then,
\begin{equation} \label{tail}
\begin{aligned}
C^{-1} \e^2 e^{-C_1 \e \abs{\x}} &\le \vtil'(\x) \le C \e^2 e^{-C_2 \e \abs{\x}}, &&\forall \x \in \RR.
\end{aligned}
\end{equation}
Therefore, as a consequence, we have
\begin{equation} \label{lower-v}
\inf_{\left[-\frac{1}{\e},\frac{1}{\e}\right]} \abs{\vtil'(\x)} \ge C^{-1}\e^2.
\end{equation}
It also holds that \(\abs{\vtil'} \sim \abs{\htil'}\) for all \(\x\in\RR\). More explicitly, we have the following:
\begin{equation} \label{ratio-vh}
\abs{\s_* \vtil'(\x) + \htil'(\x)} \le C \e \quad \forall \x \in \RR,
\end{equation}
where \(\s_* \coloneqq 1/v_-\) which satisfies
\begin{equation} \label{sm1}
\abs{\s_\e-\s_*} \le C \e.
\end{equation}
\end{lemma}
\begin{proof}
Since the proof is essentially the same as in \cite[Lemma 2.1]{KV21}, we omit the proof.
\end{proof}

The exponential decay of the viscous shock derivatives allows us to control bad terms in Section \ref{sec:out}, which is manifested in the following lemma. 
\begin{lemma}\label{lemma_pushing}
There exists a constant \(C>0\) such that for a smooth function \(f\colon \RR\to \RR\) with 
\(
\int_\RR \vtil'(\x) \abs{f(\x)} d\x <\infty,
\)
the following holds: for any \(\x_0\in[-1/\e, 1/\e]\), 
\[
\abs{\int_\RR \vtil'(\x) \int_{\x_0}^{\x} f(\z) d\z d\x} 
\le \frac{C}{\e}\int_\RR \vtil'(\z) \abs{f(\z)} d\z.
\]
\end{lemma}
\begin{proof}
The proof can be found in \cite[Lemma 4.2]{EEK-BNSF}.
\end{proof}

\subsection{Relative entropy method}
Our analysis is based on the relative entropy method, which was first introduced by Dafermos \cite{Dafermos79-ARMA} and Diperna \cite{Diperna79}. It is purely nonlinear, and allows us to handle rough and large perturbations.

To use the relative entropy method, we rewrite \eqref{main} into the following general system of viscous conservation laws:
\begin{equation} \label{VCL}
\rd_t U + \rd_\x A(U) =
\begin{pmatrix}
-\rd_\x(v^\b \rd_\x p(v)) \\
0
\end{pmatrix}
\end{equation}
where 
\[
U\coloneqq
\begin{pmatrix}
v \\ h
\end{pmatrix}, \quad
A(U)\coloneqq
\begin{pmatrix}
-\s_\e v - h \\ -\s_\e h + p(v)
\end{pmatrix}.
\]
The system \eqref{VCL} has a convex entropy \(\et(U)\coloneqq \frac{h^2}{2}-\log v\).\\
Using the derivative of the entropy as
\[
\nabla \et(U) =
\begin{pmatrix}
-p(v) \\ h
\end{pmatrix},
\]
the system above \eqref{VCL} can be rewritten as
\begin{equation} \label{VCL'}
\rd_t U + \rd_\x A(U) = \rd_\x (M(U)\rd_\x \nabla \et(U)),
\end{equation}
where \(M(U)=\begin{pmatrix} v^\b && 0 \\ 0 && 0 \end{pmatrix}\), and \eqref{shock} can be rewritten as
\begin{equation} \label{VCLshock}
\rd_\x A(\Util) = \rd_\x \left(M(\Util)\rd_\x\nabla \et(\Util)\right).
\end{equation}
Consider the relative entropy function defined by 
\[
\et(U|V) = \et(U)-\et(V)-\nabla \et(V)(U-V),
\]
which is equivalent to \eqref{eta_def}.

We consider a weighted relative entropy between the solution \(U\) of \eqref{VCL'} and the viscous shock \(\Util \coloneqq \begin{pmatrix} \vtil \\ \htil \end{pmatrix}\) in \eqref{shock} (or equivalently, \eqref{VCLshock}) up to a shift \(X(t)\):
\[
a(\x) \et (U(t,\x+X(t)|\Util(\x)))
\]
where \(a\) is a smooth weight function.
The shift \(X\) and weight \(a\) will be determined later.

In Lemma \ref{lem-rel}, we will derive a quadratic structure on \(\frac{d}{dt}\int_\RR a(\x)\et(U(t,\x+X(t)|\Util))d\x\).
For this purpose, we introduce a simple notation: for any function \(f \colon \RR^+ \times \RR \to \RR\) and the shift \(X(t)\), 
\[
f^{\pm X}(t,\x) \coloneqq f(t,\x \pm X(t)).
\]
We also introduce the function space:
\[
\Hcal \coloneqq \{(v,h) \mid  v^{-1}, v \in L^{\infty}(\RR),~ h-\htil \in L^2(\RR), \rd_\x \left(p(v)-p(\vtil)\right) \in L^2(\RR) \},
\]
on which the functionals \(Y, \Jcal^{bad},\Jcal^{para}, \Jcal^{good}\) in \eqref{ybg-first} are well-defined for all \(t\in (0,T)\).

\begin{remark}
    As mentioned in Remark \ref{rem-vh}, we consider the solution \((v,h) \in \Hcal_T\) to \eqref{main}. Then, using \(v_\x \in C(0,T;L^2(\RR))\) and \(v^{-1}, v \in C(0,T;L^\infty(\RR))\), we obtain
    \[
        \rd_\x \left(p(v)-p(\vtil)\right) \in C(0,T;L^2(\RR)), 
    \]
    which implies that \((v,h)(t) \in \Hcal\) for all \(t \in [0,T]\).
\end{remark}

\begin{lemma}\label{lem-rel}
Let \(a\colon\RR\to\RR^+\) be any positive smooth bounded function whose derivative is bounded and integrable.
Let \(\Ut\coloneqq (\vtil,\htil)\) be the viscous shock in \eqref{shock}.
For any solution \(U\in\Hcal_T\) to \eqref{VCL}, and any absolutely continuous shift \(X \colon [0,T] \to \RR\), the following holds. \\
\begin{equation}\label{ineq-0}
\begin{aligned}
&\frac{d}{dt}\int_{\RR} a(\x)\eta(U^X(t,\x)|\Ut(\x)) d\x \\
&\qquad\qquad\qquad
=\dot X(t) Y(U^X) + \Jcal^{bad}(U^X) + \Jcal^{para}(U^X) - \Jcal^{good}(U^X),
\end{aligned}
\end{equation}
where
\begin{equation}
\begin{aligned}\label{ybg-first}
Y(U) &\coloneqq -\int_\RR a'\et(U|\Util)d\x + \int_\RR a \rd_\x \nabla \et(\Util)(U-\Util) d\x, \\
\Jcal^{bad}(U) &\coloneqq \int_\RR a' (p(v)-p(\vtil))(h-\htil) d\x
+ \s_\e \int_\RR a \vtil' p(v|\vtil) d\x, \\
\Jcal^{para}(U) &\coloneqq 
- \int_\RR a' v^\b (p(v)-p(\vtil)) \rd_\x(p(v)-p(\vtil)) d\x \\
&\qquad
- \int_\RR a'(p(v)-p(\vtil)) (v^\b-\vtil^\b) \rd_\x p(\vtil) d\x \\
&\qquad
- \int_\RR a \rd_\x(p(v)-p(\vtil)) (v^\b-\vtil^\b) \rd_\x p(\vtil) d\x, \\
\Jcal^{good}(U) &\coloneqq \s_\e\int_\RR a' \Phif{v}{\vtil} d\x
+ \frac{\s_\e}{2}\int_\RR a' (h-\htil)^2 d\x
+ \int_\RR a v^\b \abs{\rd_\x (p(v)-p(\vtil))}^2 d\x.
\end{aligned}
\end{equation}
\end{lemma}

\begin{remark}
In what follows, we will define the weight function \(a\) such that \(\s_\e a' > 0\).
Therefore, \(-\Jcal^{good}\) consists of three good terms, while \(\Jcal^{bad}\) and \(\Jcal^{para}\) consist of bad terms.
\end{remark}

\begin{proof}
The proof follows the same argument as in the isentropic case.
Thus, we refer to \cite[Lemma 4.2]{KV-Inven} for details and omit the proof here.
\end{proof}

\subsection{Construction of the weight function}
We define the weight function \(a\) by
\begin{equation}\label{weight-a}
a(\xi)=1 - \frac{\lambda}{\e} \big(p(\vtil(\x))-p(v_-)\big),
\end{equation}
where \(\l>0\) is its total variation of \(a\).
We present some useful properties on the weight \(a\).
First of all, the weight function \(a\) is positive and increasing, and satisfies \(1 \le a \le 1+\l\).
Since \(p'(v_-) \le p'(\vtil) \le p'(2v_-)\) and
\begin{equation} \label{der-a}
a'(\x) = -\l \frac{p(\vtil)'}{[p]},
\end{equation}
it holds from \eqref{tail} and \eqref{ratio-vh} that 
\begin{equation} \label{der-scale}
\abs{a'} \sim \frac{\l}{\e}\abs{\vtil'} \sim \frac{\l}{\e}\abs{\htil'}, \qquad 
\abs{a'(\x)} \le C\e\l e^{-C\e\abs{\x}} \quad \text{ for all } \x\in\RR.
\end{equation}

\subsection{Maximization in terms of \(h-\htil\)}\label{sec:mini}
In order to estimate the right-hand side of \eqref{ineq-0}, we will use Proposition \ref{prop:main3}, which provides a sharp estimate with respect to \(p(v)-p(\vtil)\) when \(\abs{p(v)-p(\vtil)} \ll 1\). This requires to rewrite \(\Jcal^{bad}\) on the right-hand side of \eqref{ineq-0} only in terms of \(p(v)\) near \(p(\vtil)\), separating \(h-\htil\) from the first term of \(\Jcal^{bad}\).
Therefore, we will rewrite \(\Jcal^{bad}\) into the maximized representation in terms of \(h-\htil\) in the following lemma.
However, we will retain all terms of \(\Jcal^{bad}\) in a region \(\{p(v)-p(\vtil) > \d\}\) for small values of \(v\), since we use the estimate \eqref{bo1m} to control the first term of \(\Jcal^{bad}\) in that region.

\begin{lemma}\label{lem-max}
Let \(a\colon \RR\to\RR^+\) be as in \eqref{weight-a}, and \(\Ut=(\vtil, \htil)\) be the viscous shock in \eqref{shock}. Let \(\d>0\) be any constant.
Then, for any \(U\in \Hcal\), 
\begin{equation}\label{ineq-1}
\Jcal^{bad} (U) -\Jcal^{good} (U)= \Bcal_\d(U)- \Gcal_\d(U),
\end{equation}
where
\begin{equation}
\begin{aligned}\label{badgood}
\Bcal_\d(U) &\coloneqq \frac{1}{2\s_\e}\int_\RR a'(p(v)-p(\vtil))^2\one{p(v)\le p(\vtil)+\d} d\x
+\s_\e \int_\RR a \vtil' p(v|\vtil) d\x \\
&\qquad
+\int_\RR a'(p(v)-p(\vtil))(h-\htil)\one{p(v)>p(\vtil)+\d} d\x \\
\Gcal_\d(U) &\coloneqq \frac{\s_\e}{2}\int_\RR a'\left(h-\htil-\frac{p(v)-p(\vtil)}{\s_\e}\right)^2 \one{p(v)\le p(\vtil)+\d} d\x \\
&\qquad
+\frac{\s_\e}{2}\int_\RR a'(h-\htil)^2 \one{p(v)>p(\vtil)+\d} d\x \\
&\qquad
+\s_\e\int_\RR a' \Phif{v}{\vtil}d\x
+\int_\RR a v^\b \abs{\rd_\x (p(v)-p(\vtil))}^2 d\x.
\end{aligned}
\end{equation}
\end{lemma}
\begin{remark}\label{rem:0}
Since $\s_\e a' >0$, $-\Gcal_\d$ consists of four good terms. 
\end{remark}

\begin{proof}
The details of the proof can be found in \cite[Lemma 4.3]{KV-Inven}.
\end{proof}

\subsection{Main proposition}
The main proposition for the proof of Theorem \ref{thm_main3} is as follows.

\begin{proposition}\label{prop:main}
There exist constants \(\e_0,\d_0,\d_3 \in (0,1/2)\) such that for any \(\e<\e_0\) and \(\d_0^{-1}\e<\l<\d_0\), the following statement holds. \\
For any \(U \in \Hcal \cap \{U \mid \abs{Y(U)}\le\e^2\}\),
\begin{align*}
\Rcal(U)&\coloneqq -\frac{1}{\e^4}Y^2(U) +\Bcal_{\d_3}(U)+\d_0\frac{\e}{\l} \abs{\Bcal_{\d_3}(U)} + \d_0\frac{\e}{\l}\Bcal_1^+(U) + \Jcal^{para}(U) + \d_0\abs{\Jcal^{para}(U)} \\
&\qquad -\Gcal_h^-(U)-\frac{1}{2}\Gcal_h^+(U) -\left(1-\d_0\frac{\e}{\l}\right)\Gcal_v(U) -(1-\d_0)\Dcal(U) \le 0,
\end{align*}
where \(Y\), \(\Bcal_{\d_3}\) and \(\Jcal^{para}\) are as in \eqref{ybg-first} and \eqref{badgood}, \(\Bcal_1^+\) denotes 
\[
\Bcal_1^+(U) \coloneqq \frac{1}{2\s_\e}\int_\RR a'(p(v)-p(\vtil))^2 \one{p(v)\le p(\vtil)+\d_3} d\x,
\]
and \(\Gcal_h^-, \Gcal_h^+, \Gcal_v\) and \(\Dcal\) denote the four terms of \(\Gcal_{\d_3}\) as follows:
\begin{equation}
\begin{aligned}\label{ggd}
&\Gcal_h^-(U)\coloneqq \frac{\s_\e}{2}\int_\RR a'(h-\htil)^2 \one{p(v)>p(\vtil)+\d_3} d\x,\\
&\Gcal_h^+(U)\coloneqq \frac{\s_\e}{2}\int_\RR a'\Big(h-\htil-\frac{p(v)-p(\vtil)}{\s_\e}\Big)^2 \one{p(v)\le p(\vtil)+\d_3} d\x ,\\
&\Gcal_v(U)\coloneqq \s_\e\int_\RR a' \Phif{v}{\vtil} d\x, \qquad
\Dcal(U)\coloneqq \int_\RR a v^\b \abs{\rd_\x (p(v)-p(\vtil))}^2 d\x.
\end{aligned}
\end{equation}
\end{proposition}

\subsection{Proof of Theorem \ref{thm_main} from Proposition \ref{prop:main}}
In this subsection, we demonstrate how Proposition \ref{prop:main} implies Theorem \ref{thm_main3}, from which Theorem \ref{thm_main} follows immediately.

For any fixed \(\e>0\), we consider a continuous function \(\Phi_\e\) defined by
\begin{equation}\label{Phi-d}
\Phi_\e (y)=
\left\{
\begin{array}{ll}
\frac{1}{\e^2}, \quad &\text{if } y\le -\e^2, \\
-\frac{1}{\e^4}y, \quad &\text{if } \abs{y}\le \e^2, \\
-\frac{1}{\e^2}, \quad &\text{if } y\ge \e^2.
\end{array} \right.
\end{equation}
Let \(\e_0,\d_0,\d_3\) be the constants in Proposition \ref{prop:main}.
Then, let \(\e,\l\) be any constants such that \(0<\e<\e_0\) and \(\d_0^{-1}\e<\l<\d_0<1/2\).

We define a shift function \(X(t)\) as a solution of the nonlinear ODE:
\begin{equation}\label{X-def}
\left\{
\begin{array}{ll}
\dot X(t) = \Phi_\e (Y(U^X)) \Big(2|\Jcal^{bad}(U^X)|+2\abs{\Jcal^{para}(U^X)}+1\Big),\\
X(0)=0,
\end{array} \right.
\end{equation}
where \(Y\), \(\Jcal^{bad}\) and \(\Jcal^{para}\) are as in \eqref{ybg-first}. \\
Then, for the solution \(U \in \Hcal_T\), there exists a unique absolutely continuous shift \(X\) on \([0,T]\).
Indeed, since \(\vtil',\htil',a'\) are bounded, smooth and integrable, using \(U\in\Hcal_T\) together with the change of variable \(\x \mapsto \x-X(t)\), we find that there exist \(a,b \in L^2(0,T)\) such that
\[
\sup_{x\in\RR} \abs{F(t,x)} \le a(t) \quad ~\text{ and }~ \quad \sup_{x\in\RR} \abs{\rd_x F(t,x)} \le b(t), \quad \forall t \in [0,T],
\]
where \(F(t,X)\) denotes the right-hand side of the ODE in \eqref{X-def}.
For more details on the existence and uniqueness theory of the ODE, we refer to \cite[Lemma A.1]{CKKV}.

Although each term has different meanings in \cite{EEK-BNSF} and the present paper, the form of the inequalities provided in \cite[Proposition 4.1]{EEK-BNSF} and Proposition \ref{prop:main}, as well as the definition of the shift, is highly similar.
Thus, Theorem \ref{thm_main3} can be proved in an analogous way.

Based on \eqref{ineq-0} and \eqref{X-def}, to get the contraction estimate \eqref{cont_main3}, it is enough to prove that for almost every time \(t>0\),
\begin{equation} \label{contem0}
\begin{aligned}
&\Phi_\e(Y(U^X)) \left(2|\Jcal^{bad}(U^X)|+2\abs{\Jcal^{para}(U^X)}+1\right)Y(U^X) \\
&\qquad\qquad\qquad\qquad\qquad\qquad\qquad
+\Jcal^{bad}(U^X) +\Jcal^{para}(U^X) -\Jcal^{good}(U^X) \le 0.
\end{aligned}
\end{equation}
We define
\begin{align*}
\Fcal(U) \coloneqq &\Phi_\e(Y(U)) \left(2|\Jcal^{bad}(U)|+2\abs{\Jcal^{para}(U)}+1\right)Y(U) \\
&\qquad\qquad\qquad\qquad\qquad\qquad\qquad
+\Jcal^{bad}(U) +\Jcal^{para}(U) -\Jcal^{good}(U), \quad \forall U \in \Hcal.
\end{align*}
From \eqref{Phi-d}, we have 
\begin{equation} \label{XY}
\Phi_\e(Y)(2|\Jcal^{bad}|+2\abs{\Jcal^{para}}+1)Y \le
\left\{
\begin{array}{ll}
-2\abs{\Jcal^{bad}}-2\abs{\Jcal^{para}}, \quad \text{if } \abs{Y}\ge\e^2,\\
-\frac{1}{\e^4}Y^2, \quad \text{if } \abs{Y}\le\e^2.
\end{array} \right.
\end{equation}
Then, we proceed in a manner similar to \cite{EEK-BNSF}.
Using \eqref{XY}, \eqref{ineq-1}, Proposition \ref{prop:main} and the fact that \(\d_0<1/2\), and using \eqref{ineq-0} and \eqref{contem0}, we find that for a.e. \(t>0\), 
\begin{equation} \label{111}
\begin{aligned}
&\frac{d}{dt}\int_\RR a\et(U^X|\Ut)d\x + \d_0\frac{\e}{\l}\Gcal_v(U^X)+\d_0\Dcal(U^X)
= \Fcal(U^X)+ \d_0\frac{\e}{\l}\Gcal_v(U^X)+\d_0\Dcal(U^X) \\
&\le -|\Jcal^{bad}(U^X)|\one{\abs{Y(U^X)}\ge \e^2}
-\d_0\frac{\e}{\l}\left(\abs{\Bcal_{\d_3}(U^X)}+\Bcal_1^+(U^X)\right)\one{\abs{Y(U^X)}\le \e^2} \\
&\qquad -\d_0\abs{\Jcal^{para}(U^X)} -\frac{1}{2}\Gcal_h^+(U^X)\one{\abs{Y(U^X)}\le \e^2} \le 0.
\end{aligned}
\end{equation}
Therefore, we have 
\[
\int_\RR a\et(U^X|\Ut)d\x + \d_0\frac{\e}{\l}\int_0^t \Gcal_v(U^X) ds
+ \d_0\int_0^t \Dcal(U^X) ds
\le \int_\RR a\et(U_0|\Ut)d\x < \infty
\]
which completes \eqref{cont_main3}.

To estimate \(\abs{\dot{X}}\), we first observe that \eqref{Phi-d} and \eqref{X-def} to yields
\begin{equation} \label{contx}
\abs{\dot{X}} \le \frac{1}{\e^2} \left(2|\Jcal^{bad}(U^X)|+2\abs{\Jcal^{para}(U^X)}+1\right), \quad \text{ for a.e. } t\in(0,T). 
\end{equation}
Notice that it follows from \eqref{111} and \(1\le a \le 3/2\) by \(\d_0<1/2\) that
\begin{equation} \label{jb-cont}
\begin{aligned}
&\int_0^T \left(|\Jcal^{bad}(U^X)|\one{\abs{Y(U^X)}\ge \e^2} +\d_0\abs{\Jcal^{para}(U^X)} +\Gcal_h^+(U^X)\one{\abs{Y(U^X)}\le \e^2} \right) dt \\
&\qquad
+ \d_0\frac{\e}{\l} \int_0^T \left(\left(\abs{\Bcal_{\d_3}(U^X)}+\Bcal_1^+(U^X)\right)\one{\abs{Y(U^X)}\le \e^2}\right)dt
\le 2\int_\RR \et(U_0|\Ut) d\x.
\end{aligned}
\end{equation}
To estimate \(|\Jcal^{bad}(U^X)|\) globally in time, using \eqref{ineq-1}, \eqref{ybg-first} and \eqref{badgood} as in \cite{EEK-BNSF}, we obtain
\begin{align*}
&|\Jcal^{bad}(U^X)|
\le |\Jcal^{bad}(U^X)|\one{\abs{Y(U^X)}\ge \e^2}
+\left(\abs{\Bcal_{\d_3}(U^X)}+C\Bcal_1^+(U)+C\Gcal_h^+(U)\right)\one{\abs{Y(U^X)}\le \e^2}.
\end{align*}
This together with \eqref{contx} implies that
\begin{multline*}
\abs{\dot{X}} \le \frac{1}{\e^2}\bigg[ 2 \Big(|\Jcal^{bad}(U^X)|\Big)\one{\abs{Y(U^X)}\ge \e^2} \\
+2\Big(\abs{\Bcal_{\d_3}(U^X)}+C\Bcal_1^+(U)+C\Gcal_h^+(U)\Big) \one{\abs{Y(U^X)}\le \e^2}+2\abs{\Jcal^{para}(U^X)} + 1 \bigg].
\end{multline*}
Moreover, it follows from \eqref{jb-cont} with \(\e<\e_0\) and \(\d_0^{-1}\e < \l < \d_0 < 1/2\) that
\begin{multline*}
\int_0^{T} \Big(|\Jcal^{bad}(U^X)|\one{\abs{Y(U^X)}\ge \e^2} + \abs{\Jcal^{para}(U^X)}\\
+ \Big(\abs{\Bcal_{\d_3}(U^X)}+\Bcal_1^+(U^X)+\Gcal_h^+(U)\Big) \one{\abs{Y(U^X)}\le \e^2}\Big) dt \le C\frac{\l}{\d_0\e}\int_\RR \et(U_0|\Ut) d\x,
\end{multline*}
for some constant \(C>0\) independent of \(\d_0, \e/\l\) and \(T\).
This completes the proof of \eqref{est-shift3}. \qed

\section{Proof of Proposition \ref{prop:main}}
\setcounter{equation}{0}
This section is dedicated to the proof of Proposition \ref{prop:main}.

\subsection{Expansion in the size of the shock} \label{section-expansion}
We define the following functionals:
\begin{equation*}
\begin{split}
&\begin{aligned} 
&\Ical_g^Y(v) \coloneqq
-\frac{1}{2\s_\e^2} \int_\RR a'\abs{p(v)-p(\vtil)}^2d\x - \int_\RR a' \Phif{v}{\vtil} d\x \\
&\qquad\qquad\qquad
-\int_\RR a p(\vtil)' (v-\vtil) d\x + \frac{1}{\s_\e} \int_\RR a \htil' (p(v)-p(\vtil)) d\x
\end{aligned} \\
&\begin{aligned}
&\Ical_1(v) \coloneqq 
\frac{1}{2\s_\e} \int_\RR a' (p(v)-p(\vtil))^2 d\x,
&&\Ical_2(v) \coloneqq
\s_\e \int_\RR a \vtil' p(v|\vtil) d\x, \\
&\Gcal_v(v) \coloneqq
\s_\e \int_\RR a' \Phif{v}{\vtil} d\x,
&&\Dcal(v) \coloneqq
\int_\RR a v^\b \abs{\rd_\x (p(v)-p(\vtil))}^2 d\x.
\end{aligned}
\end{split}
\end{equation*}
Note that all of these quantities depend only on \(v\), not on \(h\).

We then present the following proposition, which provides a sharp estimate when \(p(v)\) exhibits values similar to \(p(\vtil)\).

\begin{proposition} \label{prop:main3}
For any constant \(C_2>0\), there exist constants \(\e_0,\d_3>0\) such that for any \(\e \in (0,\e_0)\) and any \(\l,\d \in (0,\d_3)\) with \(\e \le \l\), the following statement holds. \\
For any function \(v \colon \RR \to \RR^+\) such that \(\Dcal(v)+\Gcal_v(v)\) is finite, if
\begin{equation} \label{assYp}
\abs{\Ical_g^Y(v)} \le C_2 \frac{\e^2}{\l},\qquad  \norm{p(v)-p(\vtil)}_{L^\infty(\RR)}\le \d_3,
\end{equation}
then
\begin{align}
\begin{aligned}\label{redelta}
\Rcal_{\e,\d}(v)
&\coloneqq -\frac{1}{\e\d}\abs{\Ical_g^Y(v)}^2
+\Ical_1(v)+\d\frac{\e}{\l}\abs{\Ical_1(v)}
+\Ical_2(v)+\d\abs{\Ical_2(v)} \\
&\quad\quad-\left(1-\d \frac{\e}{\l} \right)\Gcal_v(v)-(1-\d)\Dcal(v)\le 0.
\end{aligned}
\end{align}
\end{proposition}

In the proof of the proposition above, the following lemma, which provides a nonlinear Poincar\`e type inequality, is needed.
\begin{proposition}\label{prop_nl_Poincare}
For a given constant \(C_1 >0\), there exists a constant \(\d_2>0\), such that for any \(\d\in(0,\d_2)\) the following statement holds. \\
For any \(W\in L^2(0, 1)\) such that \(\sqrt{y(1-y)}\rd_y W\in L^2(0, 1)\), if \(\int_0^1 \abs{W(y)}^2 dy\le C_1\), then 
\begin{align*}
\Rcal_\d(W) &\coloneqq -\frac{1}{\d}\left(\int_0^1 W^2 dy + 2\int_0^1 W dy\right)^2 
+ (1+\d)\int_0^1 W^2 dy \\
&\qquad
+ \frac{2}{3}\int_0^1 W^3 dy + \d\int_0^1\abs{W}^3 dy - (1-\d)\int_0^1 y(1-y)\abs{\rd_y W}^2 dy
\le 0.
\end{align*}
\end{proposition}

\begin{proof}
The proof can be found in \cite[Proposition 3.3]{KV21}, and \cite[Proposition 5.1]{EEK-BNSF}.
\end{proof}

We now present the proof of Proposition \ref{prop:main3} below.

\begin{proof}[Proof of Proposition \ref{prop:main3}]
For given \((v_-,u_-)\), we consider any constant \(\d_3 \in (0,1/2)\) such that 
\(
\d_3 < \d_*
\)
where the constants \(\d_*>0\) is in Lemma \ref{lem:local}.

To use Proposition \ref{prop_nl_Poincare}, we will rewrite the functionals \(\Ical_g^Y, \Ical_1, \Ical_2, \Gcal_v\) and \(\Dcal\) in terms of the following variables:
\begin{equation} \label{yandw}
y \coloneqq \frac{p(v_-)-p(\vtil(\x))}{\e}, \qquad w \coloneqq (p(v)-p(\vtil)) \circ y^{-1}, \qquad W \coloneqq \frac{\l}{\e} w.
\end{equation}
Notice that since \(p(\vtil(\x))\) is decreasing in \(\x\), we use a change of variable \(\x\in\RR \mapsto y\in[0,1]\).

Now, we express \(\Ical_g^Y, \Ical_1, \Ical_2, \Gcal_v\) and \(\Dcal\) using the variables introduced in \eqref{yandw}. \\

\(\bullet\) \textbf{Change of variable for \(\Ical_g^Y\):} We decompose \(\Ical_g^Y\) into four parts as follows and analyze them separately:
\begin{align*} 
&\Ical_g^Y(v) \coloneqq
\underbrace{-\frac{1}{2\s_\e^2} \int_\RR a'\abs{p(v)-p(\vtil)}^2d\x}_{\eqqcolon Y_1}
\underbrace{-\int_\RR a' \Phif{v}{\vtil} d\x}_{\eqqcolon Y_2} \\
&\qquad\qquad\qquad
\underbrace{-\int_\RR a p(\vtil)' (v-\vtil) d\x}_{\eqqcolon Y_3}
\underbrace{+\frac{1}{\s_\e} \int_\RR a \htil' (p(v)-p(\vtil)) d\x}_{\eqqcolon Y_4}.
\end{align*} 
For \(Y_1\) and \(Y_2\), we use \eqref{Phi-est1} and \eqref{sm1} to obtain
\[
\abs{Y_1+\frac{1}{2\s_*^2}\frac{\e^2}{\l}\int_0^1 W^2 dy}
+\abs{Y_2+\frac{1}{2\s_*^2}\frac{\e^2}{\l}\int_0^1 W^2 dy} \le C(\e_0+\d_3) \int_0^1 W^2 dy.
\]
For \(Y_3\) and \(Y_4\), we use \eqref{ratio-vh} and \eqref{sm1} to obtain
\[
\abs{Y_3+\frac{1}{\s_*^2}\frac{\e^2}{\l}\int_0^1 W dy}
+\abs{Y_4+\frac{1}{\s_*^2}\frac{\e^2}{\l}\int_0^1 W dy} \le C(\e_0+\d_3) \int_0^1 \abs{W}dy.
\]
Thus, it follows that 
\begin{equation} \label{insidey}
\abs{\s_*^2 \frac{\l}{\e^2}\Ical_g^Y + \int_0^1 (W^2 + 2W)dy}
\le C(\e_0+\d_3) \left(\int_0^1 W^2 dy + \int_0^1 \abs{W} dy \right).
\end{equation}
Moreover, this together with the assumption \eqref{assYp} implies that 
\[
\int_0^1 W^2 - 2 \abs{\int_0^1 W dy}
\le C_2 \s_*^2 + C (\e_0+\d_3) \left(\int_0^1 W^2 dy + \int_0^1 \abs{W} dy \right).
\]
Using 
\[
\abs{\int_0^1 W dy} \le \int_0^1 \abs{W} dy \le \frac{1}{8} \int_0^1 W^2 dy +8
\]
we find that for small enough \(\e_0,\d_3>0\),
\begin{align*}
\int_0^1 W^2 dy
&\le \frac{1}{2} \int_0^1 W^2 dy + C +24.
\end{align*}
Thus, there exists a constant \(C_1>0\) depending on \(C_2\) (but not on \(\e,\l\)) such that
\begin{equation} \label{L2normW}
\int_0^1 W^2 dy \le C_1.
\end{equation}
With this and by taking \(\e_0<\d_3\), we apply \eqref{insidey} to derive an estimate on \(\abs{\Ical_g^Y}^2\) as follows:
\begin{equation} \label{insidey2}
\begin{aligned}
-\s_*^2 \frac{\l^2}{\e^3} \frac{\abs{\Ical_g^Y}^2}{\e\d_3}
&= -\frac{1}{\s_*^2 \d_3}\abs{\s_*^2 \frac{\l}{\e^2}\Ical_g^Y}^2 \\
&\le -\frac{1}{2\s_*^2 \d_3} \abs{\int_0^1 (W^2+2W)dy}^2
+C \frac{(\e_0+\d_3)^2}{\d_3} \abs{\int_0^1 (W^2+2\abs{W})dy}^2 \\
&\le -\frac{1}{2\s_*^2 \d_3} \abs{\int_0^1 (W^2+2W)dy}^2
+C \d_3 \int_0^1 W^2 dy.
\end{aligned}
\end{equation}

\(\bullet\) \textbf{Change of variable for \(\Ical_1 - \Gcal_v\):}
Using \eqref{Phi-est1}, we write down \(\Ical_1\) and \(\Gcal_v\) together, and observe that they have like terms of order \(2\):
\begin{align*}
\Ical_1(v) -\Gcal_v(v) 
\le \frac{1}{2\s_\e} \int_\RR a' (1-\s_\e^2 \vtil^2) (p(v)-p(\vtil))^2 d\x
+\frac{2}{3}\s_\e\int_\RR a' \vtil^3 (p(v)-p(\vtil))^3 d\x.
\end{align*}
Note by \eqref{sm1} that \(\abs{1-\s_\e^2\vtil^2}\le C\e\) and \(\abs{1/\s_*^3 -\vtil^3}\le C\e\) for all \(\x\in\RR\).
Hence, we use \eqref{sm1} once again to obtain 
\begin{align*}
\Ical_1(v)-\Gcal_v(v)
&\le \frac{2}{3\s_*^2} \int_\RR a' (p(v)-p(\vtil))^3 d\x \\
&\qquad
+ C \e \int_\RR a' (p(v)-p(\vtil))^2 d\x
+ C \e \int_\RR a' \abs{p(v)-p(\vtil)}^3 d\x,
\end{align*}
which can be rewritten in terms of the variables in \eqref{yandw} as follows:
\begin{equation} \label{insideb1}
\begin{aligned}
\s_*^2 \frac{\l^2}{\e^3} (\Ical_1-\Gcal_v)
&\le \frac{2}{3} \int_0^1 W^3 dy 
+ C \d_3 \int_0^1 W^2 dy
+ C \e_0 \int_0^1 \abs{W}^3 dy.
\end{aligned}
\end{equation}

\(\bullet\) \textbf{Change of variable for \(\Ical_2\):}
To estimate \(\Ical_2\), we first note that \(p(v|\vtil) = v (p(v)-p(\vtil))^2\) by simple computations.
Then, from \(\vtil' = -\vtil^2 p(\vtil)'\), we obtain 
\[
\Ical_2(v) = \s_\e \int_\RR a\vtil'p(v|\vtil)d\x
=-\s_\e \int_\RR a v \vtil^2 p(\vtil)' (p(v)-p(\vtil))^2 d\x.
\]
Hence, \eqref{sm1} and \eqref{assYp} imply that 
\begin{equation} \label{insideb2}
\abs{\s_*^2 \frac{\l^2}{\e^3}\Ical_2 - \int_0^1 W^2 dy}
\le C(\e_0+\d_3)\int_0^1 W^2 dy.
\end{equation}

\(\bullet\) \textbf{Change of variable for \(\Dcal\):}
To describe \(\Dcal\) in terms of the variables in \eqref{yandw}, a certain Jacobian estimate, as addressed in Lemma \ref{lem-Jacobian} below, is needed.

With the variables \(y\) and \(W\), the diffusion \(\Dcal\) is described as 
\[
\Dcal = \frac{\e^2}{\l^2} \int_0^1 a v^\b \abs{W_y}^2 \left(\frac{dy}{d\x}\right) dy.
\]
Thanks to Lemma \ref{lem-Jacobian} below, it holds from \eqref{assYp} that
\begin{equation} \label{insided}
\abs{\s_*^2 \frac{\l^2}{\e^3}\Dcal - \int_0^1 y(1-y)\abs{W_y}^2 dy}
\le C(\e_0+\d_3) \int_0^1 y(1-y)\abs{W_y}^2 dy.
\end{equation}

\(\bullet\) \textbf{Conclusion:}
To complete the proof, we observe that for any \(\d \in(0,\d_3)\),
\begin{align*}
\Rcal_{\e,\d}(v)
&\le -\frac{1}{\e\d_3}\abs{\Ical_g^Y(v)}^2
+\Ical_1(v)+\d_3\frac{\e}{\l}\abs{\Ical_1(v)}
+\Ical_2(v)+\d_3\abs{\Ical_2(v)} \\
&\quad\quad-\left(1-\d_3 \frac{\e}{\l} \right)\Gcal_v(v)-(1-\d_3)\Dcal(v).
\end{align*}
Then synthesizing \eqref{insidey2}, \eqref{insideb1}, \eqref{insideb2} and \eqref{insided}, and taking into account \(\e_0<\d_3\), we obtain the following: for any \(\d \in (0,\d_3)\) and for some constants \(C_3,C_4>0\),
\begin{align*}
\s_*^2 \frac{\l^2}{\e^3} \Rcal_{\e,\d}(v)
&\le -\frac{1}{C_3 \d_3} \abs{\int_0^1 (W^2+2W)dy}^2 + \int_0^1 W^2 dy + C_4 \d_3 \int_0^1 W^2 dy \\
&\qquad
+\frac{2}{3} \int_0^1 W^3 dy + C_4 \d_3 \int_0^1 \abs{W}^3 dy
-(1-C_4\d_3)\int_0^1 y(1-y)\abs{W_y}^2 dy.
\end{align*}
We then fix \(\d_2\) of Proposition \ref{prop_nl_Poincare} corresponding to \(C_1\) in \eqref{L2normW}, and choose a small enough \(\d_3\) such that \(\max(C_3,C_4)\d_3 < \d_2\).
Hence, Proposition \ref{prop_nl_Poincare} gives the desired inequality \eqref{redelta}.
\end{proof}

We here present the Jacobian estimate which was used to rewrite the diffusion term.
\begin{lem} \label{lem-Jacobian}
Consider the variables in \eqref{yandw}. Then the following estimate holds.
\[
\abs{\frac{\vtil^\b}{y(1-y)}\frac{dy}{d\x} - \frac{\e}{\s_*^2}} \le C \e^2.
\]
\end{lem}

\begin{proof}
From \eqref{yandw}, we have 
\[
\frac{1}{y(1-y)}\frac{dy}{d\x} = \frac{1}{y}\frac{dy}{d\x} + \frac{1}{1-y}\frac{dy}{d\x} = \frac{\rd_\x p(\vtil)}{p(\vtil) - p_-} - \frac{\rd_\x p(\vtil)}{p(\vtil) - p_+}. 
\]
On the other hands, integrating the viscous shock equations \eqref{shock} over \((\pm\infty,\x]\) gives that 
\[
\left\{
\begin{array}{ll}
- \s_\e (\vtil - v_\pm) - (\htil - h_\pm) = - \vtil^\b \rd_\x p(\vtil), \\
- \s_\e (\htil - h_\pm) + (p(\vtil) - p_\pm) = 0.
\end{array} \right.
\]
Then, we have 
\[
\vtil^\b \rd_\x p(\vtil)
= \s_\e (\vtil - v_\pm) + (\htil - h_\pm)
= \frac{1}{\s_\e} \left( \s_\e^2 (\vtil - v_\pm) + (p(\vtil) - p_\pm)\right).
\]
Using \eqref{end-con}, we obtain
\begin{align*}
\frac{\vtil^\b}{y(1-y)}\frac{dy}{d\x} &= \frac{\vtil^\b \rd_\x p(\vtil)}{p(\vtil) - p_-} - \frac{\vtil^\b \rd_\x p(\vtil)}{p(\vtil) - p_+} = \s_\e \left( \frac{\vtil-v_-}{p(\vtil)-p_-} - \frac{\vtil-v_+}{p(\vtil)-p_+} \right) \\
&= \s_\e(- \vtil v_- + \vtil v_+) = \frac{\e}{\s_\e}\vtil.
\end{align*}
Thus, using \eqref{sm1}, we obtain the desired estimate.
\end{proof}


\subsection{Truncation of the big values of \(\abs{p(v)-p(\vtil)}\)}\label{sec:out}
In order to apply Proposition \ref{prop:main3} in the proof of Proposition \ref{prop:main}, we need to show that all bad terms for values of \(v\) such that \(\abs{p(v)-p(\vtil)}\ge\d_3\) have a small effect.
However, the value of \(\d_3\) in Proposition \ref{prop:main3} is itself conditioned to the constant \(C_2\) in the proposition.
Thus, we first need to find a uniform bound on \(\Ical_g^Y\) that is independent of the truncation size \(k\).

We consider a truncation on \(\abs{p(v)-p(\vtil)}\) with a constant \(k>0\).
Later we will consider the case \(k=\d_3\) as in Proposition \ref{prop:main3}.
However, for now, we consider the general case \(k\) to estimate the constant \(C_2\).
To this end, let \(\psi_k\) be a continuous function defined by
\[
\psi_k(y) = \inf (k, \sup(-k,y)), \quad k>0.
\]
We then define the function \(\vbar_k\) uniquely (since the function \(p\) is one to one) as 
\begin{equation} \label{vbar-k}
p(\vbar_k)-p(\vtil) = \psi_k(p(v)-p(\vtil)).
\end{equation}

We have the following lemma.
\begin{lemma} \label{lemmeC2}
For a fixed \((v_-,u_-)\in \RR^+ \times \RR\), there exist constants \(C_2,k_0,\e_0,\d_0>0\) such that for any \(\e \le \e_0, \e/\l \le \d_0\) with \(\l<1/2\), the followings hold whenever \(\abs{Y(U)}\le\e^2:\)
\begin{align}
&\int_\RR a' \Phif{v}{\vtil} d\x + \int_\RR a' (h-\htil)^2 d\x \le C\frac{\e^2}{\l}, \label{locE} \\
&\abs{\Ical_g^Y(\vbar_k)} \le C_2 \frac{\e^2}{\l}, \qquad \text{for every } k\le k_0. \label{lbis}
\end{align}
\end{lemma}

\begin{proof}
Since $Y(U)$ consists of the localized quadratic perturbations as in \eqref{locE} and the localized linear perturbations, the assumption \(\abs{Y(U)}\le\e^2\) would imply the smallness \eqref{locE}, which gives \eqref{lbis}.
The proof follows the same argument as the proof of \cite[Lemma 4.4]{KV-Inven}, which utilizes the quadratic structure \eqref{rel_Phi}\(_1\) on a compact set, and the linear structure \eqref{rel_Phi}\(_2\) on its complement with the scale difference \eqref{der-scale}.
Thus, we omit the details.
\end{proof}

Now we fix the constant \(\d_3\) of Proposition \ref{prop:main3} associated with the constant \(C_2\) of Lemma \ref{lemmeC2}.
Without loss of generality, we may assume \(\d_3<k_0\) (since Proposition \ref{prop:main3} holds for any smaller \(\d_3\)).
From now on, we set (without confusion) 
\[
\vbar \coloneqq \vbar_{\d_3}, \quad
\Ubar \coloneqq (\vbar,h), \quad
\Bcal \coloneqq \Bcal_{\d_3}, \quad
\Gcal \coloneqq \Gcal_{\d_3}.
\]
In what follows, for simplification, we use the following notation:
\[
\O \coloneqq \{\x|(p(v)-p(\vtil))(\x)\le\d_3\}.
\]
Note by Lemma \ref{lemmeC2} that
\[
\abs{\Ical_g^Y(\vbar)} \le C_2 \frac{\e^2}{\l}.
\]

We also consider one-sided truncations which allow us to control the bad terms in different ways for each case of small or large values of \(v\).
Thus, we define \(\vbar_s\) and \(\vbar_b\) as follows:
\begin{equation} \label{trunc-sb}
p(\vbar_s)-p(\vtil) \coloneqq \psi_{\d_3}^b(p(v)-p(\vtil)), \quad
p(\vbar_b)-p(\vtil) \coloneqq \psi_{\d_3}^s(p(v)-p(\vtil)),
\end{equation}
where \(\psi_{\d_3}^b\) and \(\psi_{\d_3}^s\) are one-sided truncations, i.e., 
\[
\psi_{\d_3}^b(y) \coloneqq \sup (-\d_3,y), \quad
\psi_{\d_3}^s(y) \coloneqq \inf (\d_3,y).
\]
Notice that the truncation \(\vbar_s\) (resp. \(\vbar_b\)) represents the truncation of big (resp. small) values of \(v\) corresponding to \(\abs{p(v)-p(\vtil)}\ge\d_3\).
We also note that
\begin{align*}
    p(\vbar_s) - p(v)
&= (\psi_{\d_3}^b - I) (p(v)-p(\vtil))
= (-(p(v)-p(\vtil))-\d_3)_+, \\
p(v) - p(\vbar_b)
&= (I-\psi_{\d_3}^s)(p(v)-p(\vtil)) 
= ((p(v)-p(\vtil))-\d_3)_+, \\
\abs{p(v)-p(\vbar)}
&= \abs{(I-\psi_{\d_3})(p(v)-p(\vtil))} 
= (\abs{p(v)-p(\vtil)}-\d_3)_+.
\end{align*}
Therefore, using \eqref{vbar-k} and \eqref{trunc-sb}, we have
\begin{equation} \label{keyD}
\begin{aligned}
\Dcal(U)
&= \int_\RR a v^\b \abs{\rd_\x (p(v)-p(\vtil))}^2 (\one{\abs{p(v)-p(\vtil)}\le\d_3}+\one{p(v)-p(\vtil)<-\d_3} +\one{p(v)-p(\vtil)>\d_3})d\x\\
&= \Dcal(\Ubar)
+\int_\RR a v^\b \abs{\rd_\x (p(v)-p(\vbar_b))}^2 d\x
+\int_\RR a v^\b \abs{\rd_\x (p(v)-p(\vbar_s))}^2 d\x \\
&\ge \int_\RR a v^\b \abs{\rd_\x (p(v)-p(\vbar_b))}^2 d\x
+\int_\RR a v^\b \abs{\rd_\x (p(v)-p(\vbar_s))}^2 d\x.
\end{aligned}
\end{equation}
This shows the monotonicity property \(\Dcal(U)\ge\Dcal(\Ubar)\).
It also holds by \eqref{Phi-sim} and \eqref{locE} that
\begin{equation} \label{l2}
0 \le \s_\e \int_\RR a' \Big(\Phif{v}{\vtil}-\Phif{\vbar}{\vtil}\Big) d\x= \Gcal_v(U)-\Gcal_v(\Ubar)
\le \Gcal_v(U) \le C \frac{\e^2}{\l}.
\end{equation}

\subsubsection{Control of hyperbolic terms}\label{section_hyperbolic}
In this subsection, we will show Proposition \ref{prop_hyperbolic_out} which provides the controls of the hyperbolic bad terms \(\Bcal_{\d_3}\) in \eqref{badgood}.
First of all, we introduce the following notations:
\[
\Bcal_{\d_3} = \Bcal_1^+ +\Bcal_1^- + \Bcal_2,
\]
where
\begin{align*}
\Bcal_1^-(U) &\coloneqq \int_{\O^c} a'(p(v)-p(\vtil))(h-\htil) d\x,
&&\Bcal_1^+(U) \coloneqq \frac{1}{2\s_\e}\int_{\O} a' (p(v)-p(\vtil))^2 d\x, \\
\Bcal_2(U) &\coloneqq \s_\e \int_\RR a \vtil' p(v|\vtil) d\x.
\end{align*}
We also recall the notations \(\Gcal_h^+,\Gcal_h^-,\Gcal_v\) and \(\Dcal\) in \eqref{ggd} for the good terms.

We then present the following proposition of the controls of the hyperbolic bad terms.
\begin{proposition}\label{prop_hyperbolic_out}
There exist constants $\e_0, \deo, C, C^*>0$ (in particular, $C, C^*$ depends on the constant $\d_3$) such that for any $\e<\e_0$ and $\d_0^{-1}\e<\lambda<\d_0<1/2$, the following statements hold.
For any $U$ such that $\abs{Y(U)}\le \e^2$,
\begin{align}
&\abs{\Bcal_1^+(U)-\Ical_1(\vbar)}
\le C\frac{\e}{\l}\Dcal(U)+C\frac{\e^6}{\l^4}\Gcal_v(U)
\label{bo1p}\\
&\abs{\Bcal_1^-(U)}
\le \d_3 \Gcal_h^-(U) + C \Big(\frac{\e}{\l}\Big)^{\frac{3}{4}} \Dcal(U) + C\frac{\e^6}{\l^4}\Gcal_v(U)
\label{bo1m}\\
&\abs{\Bcal_2(U)-\Bcal_2(\Ubar)}
\le C\frac{\e^2}{\l^2} \Dcal(U) +C\frac{\e^7}{\l^5} \Gcal_v(U) +C\frac{\e}{\l}(\Gcal_v(U)-\Gcal_v(\Ubar))
\label{bo2}\\
&\abs{\Bcal_{\d_3}(U)}+\abs{\Bcal_1^+(U)}\le C^*\Big(\frac{\e^2}{\l} + \Big(\frac{\e}{\l}\Big)^{\frac{3}{4}} \Dcal(U)\Big)
\label{bo}
\end{align}
\end{proposition}

To prove the proposition, we require a set of pointwise estimates, as presented below.
The lemma below plays a crucial role in our further analysis.
\begin{lemma} \label{lem-pw}
Under the same assumption as Proposition \ref{prop_hyperbolic_out}, there exists \(\x_0 \in [-1/\e, 1/\e]\) such that the followings hold: for any \(\x\in\RR\),
\begin{align}
&\abs{p(v)-p(\vbar_s)(\x)} \le C\sqrt{\int_{\x_0}^\x \one{p(v)-p(\vtil)<-\d_3}d\z}\sqrt{\Dcal(U)} \label{pwp1}\\
&\abs{v^{\b/2}(p(v)-p(\vbar_b))(\x)}
\le C\bigg(\sqrt{\int_{\x_0}^\x \one{p(v)-p(\vtil)>\d_3}d\z}\sqrt{\Dcal(U)} + \sqrt{\frac{\e^4}{\l^3}}\sqrt{\Gcal_v(U)}\bigg) \label{pwp2}\\
&\abs{v^{\b/2}(p(v)-p(\vbar_b))(\x)} 
\le C\Big(\sqrt{\abs{\x}+\frac{1}{\e}}\sqrt{\Dcal(U)} + \sqrt{\frac{\e^4}{\l^3}}\sqrt{\Gcal_v(U)}\Big) \label{pwp2'}
\end{align}
\end{lemma}
\begin{proof}
Using \eqref{lower-v}, \eqref{der-scale} and \eqref{locE}, we get 
\[
2\e \int_{-1/\e}^{1/\e} \Phif{v}{\vtil} d\x
\le \frac{2\e}{\inf_{[-1/\e,1/\e]} a' } \int_\RR a' \Phif{v}{\vtil}d\x
\le C \frac{\e}{\l\e}\frac{\e^2}{\l} \le C \Big(\frac{\e}{\l}\Big)^2.
\]
Then, there exists \(\x_0 \in [-1/\e,1/\e]\) such that \(\Phi(v(\x_0)/\vtil(\x_0))\le C(\e/\l)^2\).
For \(\d_0>0\) small enough, and using \eqref{p-quad}, we have 
\[
\abs{(p(v)-p(\vtil))(\x_0)} \le C\frac{\e}{\l}.
\]
Thus, if \(\d_0\) is small enough such that \(C\e/\l \le \d_3/2\), we find from the definition of \(\vbar\) that 
\[
\abs{(p(v)-p(\vbar))(\x_0)}=0.
\]
It also holds by \eqref{trunc-sb} that
\[
\abs{(p(v)-p(\vbar_s))(\x_0)}=\abs{(p(v)-p(\vbar_b))(\x_0)}=0.
\]
Therefore, using \eqref{keyD}, we establish \eqref{pwp1}: for any \(\x\in\RR\),
\begin{align*}
\abs{p(v)-p(\vbar_s)(\x)}
&\le \abs{\int_{\x_0}^\x \rd_\z (p(v)-p(\vbar_s))(\z) d\z}
\le C \abs{\int_{\x_0}^\x v^{\b/2} \rd_\z (p(v)-p(\vbar_s))(\z) d\z} \\
&\le C \sqrt{\int_{\x_0}^\x \one{p(v)-p(\vtil)<-\d_3}(\z) d\z} \sqrt{\Dcal(U)}.
\end{align*}

To prove \eqref{pwp2}, we use the fundamental theorem of calculus as below:
\begin{equation} \label{bb1-pwp2}
\begin{aligned}
&\abs{v^{\b/2} (p(v)-p(\vbar_b))(\x)} 
= \abs{\frac{1}{p(v)^{\b/2}} (p(v)-p(\vbar_b))(\x)} \\
&\le \abs{\int_{\x_0}^\x v^{\b/2} \rd_\z(p(v)-p(\vbar_b)) d\z}
+ \frac{\b}{2}\abs{\int_{\x_0}^\x \frac{(p(v)-p(\vbar_b))}{p(v)^{\b/2 +1}} \rd_\z(p(v)-p(\vtil)) d\z} \\
&\qquad\qquad
+ \frac{\b}{2}\abs{\int_{\x_0}^\x \frac{(p(v)-p(\vbar_b))}{p(v)^{\b/2 +1}} \rd_\z(p(\vtil)) d\z} \\
&\le \abs{\int_{\x_0}^\x v^{\b/2} \rd_\z(p(v)-p(\vbar_b)) d\z}
+ C\abs{\int_{\x_0}^\x v^{\b/2} \rd_\z(p(v)-p(\vtil)) \one{p(v)-p(\vtil)>\d_3} d\z} \\
&\qquad\qquad
+ C\abs{\int_{\x_0}^\x \vtil' \one{p(v)-p(\vtil)>\d_3} d\z}.
\end{aligned}
\end{equation}
To handle the first two terms on the right-hand side, we use H\"older's inequality and \eqref{keyD} as above.
However, for the third term, we need the following observation: for any \(\x\in\RR\) with \(\abs{(p(v)-p(\vbar))(\x)}>0\) or equivalently \(\abs{(p(v)-p(\vtil))(\x)}>\d_3\), we have \(\Phi(v/\vtil)(\x)\ge\a\) for some constant \(\a>0\) depending only on \(\d_3\).
Thus, we have
\[
\one{\abs{p(v)-p(\vbar)}>0} = \one{\abs{p(v)-p(\vtil)}>\d_3}
\le \frac{\Phi(v/\vtil)}{\a}.
\]
In particular, we have
\begin{equation} \label{Qconst-sb}
\begin{aligned}
\one{\abs{p(v)-p(\vbar_s)}>0}+\one{\abs{p(v)-p(\vbar_b)}>0}
= \one{\abs{p(v)-p(\vtil)}<\d_3}
&\le \frac{\Phi(v/\vtil)}{\a}.
\end{aligned}
\end{equation}
Now we use \eqref{der-scale}, \eqref{Qconst-sb} and \eqref{locE} in the following manner: 
\begin{equation} \label{bb2-pwp2}
\abs{\int_{\x_0}^\x \vtil' \one{p(v)-p(\vtil)>\d_3} d\z}
\le C\frac{\e}{\l} \int_\RR a' \Phif{v}{\vtil} d\x
\le C\sqrt{\frac{\e^4}{\l^3}}\sqrt{\Gcal_v(U)}.
\end{equation}
Thus, combining \eqref{bb1-pwp2} and \eqref{bb2-pwp2}, we obtain \eqref{pwp2} and \eqref{pwp2'} as we desired.
\end{proof}

We are now prepared to prove Proposition \ref{prop_hyperbolic_out}.
\begin{proof}[Proof of Proposition \ref{prop_hyperbolic_out}]

\bpf{bo1p}
From the definitions of \(\Bcal_1^+\) and \(\Ical_1\), we observe
\begin{align*}
&\abs{\Bcal_1^+(U)-\Ical_1(\vbar)}
=\frac{1}{2\s_\e}\abs{\int_\O a' \abs{p(v)-p(\vtil)}^2 d\x - \int_\RR a'\abs{p(\vbar)-p(\vtil)}^2 d\x} \\
&\le C\int_\O a' \abs{p(v)-p(\vbar)}^2 d\x
+C\int_\O a' \abs{p(v)-p(\vbar)}\abs{p(\vbar)-p(\vtil)} d\x \\
&\qquad\qquad
+C\int_{\O^c} a' \abs{p(\vbar)-p(\vtil)}^2 d\x \\
&\le C\int_\O a' \abs{p(v)-p(\vbar)}^2 d\x
+C\int_\O a' \abs{p(v)-p(\vbar)} d\x
+C\int_{\O^c} a' d\x
\eqqcolon I_1+I_2+I_3.
\end{align*}
For \(I_1\), since the region of the integration is \(\O\), \eqref{pwp1} gives the following bound:
\[
I_1 = \int_\RR a' \abs{p(v)-p(\vbar_s)}^2 d\x
\le C \Dcal(U) \int_\RR a' \int_{\x_0}^\x \one{p(v)-p(\vtil)<-\d_3} d\z d\x.
\]
Then, using Lemma \ref{lemma_pushing}, \eqref{Qconst-sb} and \eqref{locE} yields that 
\begin{equation} \label{bo1pJ1}
I_1
\le C \Dcal(U) \frac{1}{\e} \int_\RR a' \Phif{v}{\vtil} d\x
\le C \frac{\e}{\l} \Dcal(U).
\end{equation}
For \(I_2\), we will perform a nonlinearization in the following sense; we temporarily use the truncations of \(v\) by \(\d_3/2\) instead of \(\d_3\), which are denoted by \((\vbar_s)_{\d_3/2}\) and \((\vbar_b)_{\d_3/2}\), i.e.,
\[
p((\vbar_s)_{\d_3/2})-p(\vtil) \coloneqq \psi_{\d_3/2}^b(p(v)-p(\vtil)), \quad
p((\vbar_b)_{\d_3/2})-p(\vtil) \coloneqq \psi_{\d_3/2}^s(p(v)-p(\vtil)).
\]
Then, since \((y-\d_3/2)_+ \ge \d_3/2\) whenever \((y-\d_3)_+ > 0\), we have
\begin{equation} \label{NL}
(y-\d_3)_+ \le (y-\d_3/2)_+ \one{y-\d_3>0} \le (y-\d_3/2)_+ \Big(\frac{(y-\d_3/2)_+}{\d_3/2}\Big)
\le \frac{2}{\d_3}(y-\d_3/2)^2_+.
\end{equation}
This implies the following pointwise inequality: 
\[
\abs{p(v)-p(\vbar_s)} \one{p(v)-p(\vtil)<-\d_3}
\le C \abs{p(v)-p((\vbar_s)_{\d_3/2})}^2.
\]
Hence, it follows from \eqref{bo1pJ1} that 
\[
I_2 = C\int_\O a' \abs{p(v)-p(\vbar_s)} d\x
\le C \int_\O a' \abs{p(v)-p((\vbar_s)_{\d_3/2})}^2 d\x
\le C \frac{\e}{\l} \Dcal(U).
\]
For \(I_3\), we use \eqref{NL} to obtain the following pointwise inequality:
\begin{equation} \label{NL-B1}
\mathbf{1}_{\O^c}
= \one{p(v)-p(\vtil)>\d_3}
\le C v^\b \abs{p(v)-p((\vbar_b)_{\d_3/2})}^2.
\end{equation}
This together with \eqref{pwp2} implies that 
\begin{align*}
I_3
&\le C \int_\RR a' v^\b\abs{p(v)-p((\vbar_b)_{\d_3/2})}^2 d\x \\
&\le C \Dcal(U) \int_\RR a' \int_{\x_0}^\x \one{p(v)-p(\vtil)>\d_3/2} d\z d\x
+C\frac{\e^4}{\l^3} \Gcal_v(U) \int_\RR a' \one{p(v)-p(\vtil)>\d_3/2} d\x.
\end{align*}
Then, we use Lemma \ref{lemma_pushing}, \eqref{Qconst-sb} and \eqref{locE} again to have 
\begin{equation} \label{bo1pJ3}
I_3 \le C\frac{\e}{\l}\Dcal(U)+C\frac{\e^6}{\l^4}\Gcal_v(U).
\end{equation}
Summing up all, we obtain the desired inequality \eqref{bo1p}.

\bpf{bo1m}
First of all, we separate \(\Bcal_1^-(U)\) into three parts:
\begin{align*}
\Bcal_1^-(U)
&= \int_{\O^c} a'(p(v)-p(\vtil))(h-\htil)\one{p(v)-p(\vtil)\le 2\d_3} d\x \\
&\qquad
+\int_{\O^c} a'(p(v)-p(\vbar_b))(h-\htil)\one{p(v)-p(\vtil)> 2\d_3} d\x \\
&\qquad
+\int_{\O^c} a'(p(\vbar_b)-p(\vtil))(h-\htil)\one{p(v)-p(\vtil)> 2\d_3} d\x
\eqqcolon J_1+J_2+J_3.
\end{align*}
For \(J_1\) and \(J_3\), since \( \abs{p(\vbar_b)-p(\vtil)}\le\d_3\) on \(\O^c\), we use Young's inequality and \eqref{bo1pJ3} to find 
\[
J_1+J_3 \le 3\d_3 \int_{\O^c} a' |h-\htil|d\x
\le \d_3 \Gcal_h^-(U) + C \int_{\O^c} a' d\x
\le \d_3 \Gcal_h^-(U) + C\frac{\e}{\l}\Dcal(U)+C\frac{\e^6}{\l^4}\Gcal_v(U).
\]
Now, it remains to control \(J_2\).
To this end, we first observe that 
\begin{align*}
    \frac{p(v)^\b}{p(v)-p(\vbar_b)} \one{p(v)-p(\vtil)>2\d_3} &\leq C \Big(\frac{p(v)}{p(v)-p(\vbar_b)} \Big)\one{p(v)-p(\vtil)>2\d_3}\\
&=C\Big(1+\frac{p(\vbar_b)}{p(v)-p(\vbar_b)}\Big) \one{p(v)-p(\vtil)>2\d_3}
\le C \one{p(v)-p(\vtil)>2\d_3}.
\end{align*}
Then, we have
\begin{equation} \label{powerup}
(p(v)-p(\vbar_b))\one{p(v)-p(\vtil)>2\d_3} \le C v^\b(p(v)-p(\vbar_b))^2.
\end{equation}
With this, using H\"older's inequality and \eqref{locE}, we find that 
\begin{align*}
J_2 &\le \Big(\int_\RR a' (p(v)-p(\vbar_b))^2 \one{p(v)-p(\vtil)>2\d_3} d\x \Big)^{1/2}
\Big(\int_\RR a' (h-\htil)^2 d\x \Big)^{1/2} \\
&\le C\sqrt{\frac{\e^2}{\l}} \Big(\int_\RR a' v^{2\b} (p(v)-p(\vbar_b))^4 d\x \Big)^{1/2}.
\end{align*}
Hence, \eqref{pwp2'} with \eqref{Qconst-sb} and \eqref{locE} implies that
\begin{equation} \label{bo1mJ2}
\begin{aligned}
J_2
&\le C \sqrt{\frac{\e^2}{\l}}\Big(\Dcal(U)^2\int_\RR a' \Big(\abs{\x}+\frac{1}{\e}\Big)^2 \one{p(v)-p(\vtil)>\d_3} d\x \Big)^{\frac{1}{2}} \\
&\qquad
+C \sqrt{\frac{\e^2}{\l}} \Big(\frac{\e^8}{\l^6}\Gcal_v(U)^2\int_\RR a' \one{p(v)-p(\vtil)>\d_3} d\x\Big)^{\frac{1}{2}} \\
&\le C \sqrt{\frac{\e^2}{\l}} \Dcal(U)
\Big(\int_\RR a' \Big(\abs{\x}+\frac{1}{\e}\Big)^2 \one{p(v)-p(\vtil)>\d_3} d\x \Big)^{\frac{1}{2}}
+ C \frac{\e^6}{\l^4} \Gcal_v(U).
\end{aligned}
\end{equation}
To control the right-hand side, the following estimate is required: thanks to \eqref{der-scale},
\begin{align*}
\int_{\abs{\x}\ge\frac{1}{\e}\left(\frac{\l}{\e}\right)^{\frac{1}{4}}} a' \abs{\x}^2 d\x
&\le C \e\l \int_{\abs{\x}\ge\frac{1}{\e}\left(\frac{\l}{\e}\right)^{\frac{1}{4}}} e^{-C\e\abs{\x}}\abs{\x}^2 d\x \\
&\le C \frac{\l}{\e^2} \int_{\abs{z}\ge\left(\frac{\l}{\e}\right)^{\frac{1}{4}}} z^2 e^{-Cz} dz \\
&\le C \frac{\l}{\e^2} \int_{\abs{z}\ge\left(\frac{\l}{\e}\right)^{\frac{1}{4}}} e^{-\frac{C}{2}z} dz
\le C \frac{\l}{\e^2} e^{-\frac{C}{2}\left(\frac{\l}{\e}\right)^{\frac{1}{4}}}
\le C \frac{\l}{\e^2} \left(\frac{\e}{\l}\right)^\frac{3}{2}.
\end{align*}
Hence, we use \eqref{Qconst-sb} and \eqref{locE} to obtain 
\begin{align*}
\int_\RR a' \Big(\abs{\x}+\frac{1}{\e}\Big)^2 \one{p(v)-p(\vtil)>\d_3} d\x
&\le C \int_{\abs{\x}\ge\frac{1}{\e}\left(\frac{\l}{\e}\right)^{\frac{1}{4}}} a' \abs{\x}^2 d\x
+C \frac{1}{\e^2}\Big(\frac{\l}{\e}\Big)^{\frac{1}{2}} \int_\RR a' \Phif{v}{\vtil} d\x \\
&\le C\sqrt{\frac{1}{\l\e}} + C\sqrt{\frac{1}{\l\e}}
= C\sqrt{\frac{1}{\l\e}}.
\end{align*}
This together with \eqref{bo1mJ2} gives that 
\[
J_2 \le C \Big(\frac{\e}{\l}\Big)^{\frac{3}{4}} \Dcal(U) + C\frac{\e^6}{\l^4}\Gcal_v(U).
\]
Summing up all, we obtain the desired inequality \eqref{bo1m}.

\bpf{bo2} Recall the definition of \(\Bcal_2\), and we use \eqref{der-scale} to get 
\begin{equation} \label{B2}
\abs{\Bcal_2(U)-\Bcal_2(\Ubar)}
\le C\frac{\e}{\l} \int_\RR a'(p(v|\vtil)-p(\vbar|\vtil)) d\x.
\end{equation}
Note by \eqref{rel_Phi-L} that
\begin{align*}
\abs{p(v|\vtil)-p(\vbar|\vtil)}
&=\abs{p(v)-p(\vbar)-p'(\vtil)(v-\vbar)}
\le \abs{p(v)-p(\vbar)} + C \abs{v-\vbar}\\
&\le \abs{p(v)-p(\vbar)} + C \Big(\Phif{v}{\vtil}-\Phif{\vbar}{\vtil}\Big).
\end{align*}
Thus, it suffices to establish a bound to the following quantity:
\[
\int_\RR a' \abs{p(v)-p(\vbar)} d\x + C\int_\RR a' \Big(\Phif{v}{\vtil}-\Phif{\vbar}{\vtil}\Big) d\x
\eqqcolon K_1+K_2.
\]
For \(K_1\), we split it into three parts: 
\begin{align*}
K_1 &= \int_\RR a' \abs{p(v)-p(\vbar_s)} d\x
+ \int_\RR a' \abs{p(v)-p(\vbar_b)} \one{p(v)-p(\vtil) > 2\d_3} d\x\\
&\qquad
+ \int_\RR a' \abs{p(v)-p(\vbar_b)} \one{p(v)-p(\vtil) \le 2\d_3} d\x
\eqqcolon K_{11}+K_{12}+K_{13}.
\end{align*}
Note that \(K_{11}\) is exactly the same as \(I_2\) in the proof of \eqref{bo1p}: 
\[
K_{11} \le C \frac{\e}{\l} \Dcal(U).
\]
For \(K_{12}\) and \(K_{13}\), we proceed as follows: we use \eqref{powerup} for \(K_{12}\).
Moreover, since \(p(v)-p(\vbar_b)=(p(v)-p(\vbar_b))\one{p(v)-p(\vtil)>\d_3}\), we use \eqref{NL-B1} for \(K_{13}\).
Thus, we obtain
\[
K_{12} + K_{13} \le C \int_\RR a' v^\b(p(v)-p(\vbar_b))^2 d\x.
\]
Then, using \eqref{pwp2} and then using Lemma \ref{lemma_pushing}, \eqref{Qconst-sb} and \eqref{locE} yields that
\begin{align*}
K_{12}+K_{13} &\le
C \Dcal(U) \int_\RR a' \int_{\x_0}^\x \one{p(v)-p(\vtil)>\d_3} d\z d\x
+C\frac{\e^4}{\l^3} \Gcal_v(U) \int_\RR a' \one{p(v)-p(\vtil)>\d_3} d\x\\
&\le C\frac{\e}{\l} \Dcal(U) +C\frac{\e^6}{\l^4} \Gcal_v(U).
\end{align*}
On the other hand, for \(K_2\), it simply holds that 
\[
K_2 \le C (\Gcal_v(U)-\Gcal_v(\Ubar)).
\]
Therefore, combining \(K_{11}\), \(K_{12}\), \(K_{13}\) and \(K_2\), we obtain
\[
\int_\RR a'(p(v|\vtil)-p(\vbar|\vtil)) d\x \le C\frac{\e}{\l} \Dcal(U) +C\frac{\e^6}{\l^4} \Gcal_v(U)
+C (\Gcal_v(U)-\Gcal_v(\Ubar)).
\]
This together with \eqref{B2} implies the desired inequality \eqref{bo2}.

\bpf{bo} Using \eqref{locE} with \eqref{p-quad}, \eqref{p-est1} and \eqref{der-scale}, we have 
\[
\abs{\Ical_1(\vbar)}+\abs{\Bcal_2(\Ubar)} \le \Big(C+C\frac{\e}{\l}\Big) \int_\RR a'\Phif{v}{\vtil}d\x
\le C\frac{\e^2}{\l}.
\]
This together with \eqref{bo1p}, \eqref{bo1m} and \eqref{bo2} gives the following:
\begin{align*}
\abs{\Bcal_{\d_3}(U)}+\abs{\Bcal_1^+(U)}
&\le 2\abs{\Bcal_1^+(U)-\Ical_1(\vbar)} +2\abs{\Ical_1(\vbar)} +\abs{\Bcal_1^-(U)} \\
&\qquad
+\abs{\Bcal_2(U)-\Bcal_2(\Ubar)} +\abs{\Bcal_2(\Ubar)}\\
&\le C \Big(\frac{\e}{\l}\Big)^{\frac{3}{4}} \Dcal(U)
+ C\frac{\e^6}{\l^4}\Gcal_v(U)
+ C\frac{\e}{\l}(\Gcal_v(U)-\Gcal_v(\Ubar)) + C\frac{\e^2}{\l}.
\end{align*}
Then, this with \eqref{locE} and \eqref{l2} establishes the desired inequality \eqref{bo}.
\end{proof}

\subsubsection{Control of parabolic terms} \label{section-parabolic}
In this subsection, we will control the parabolic bad terms \(\Jcal^{para}\) of \eqref{ybg-first}.
For this purpose, we introduce the following notations:
\[
\Jcal^{para} = \Bcal_3 + \Bcal_4 + \Bcal_5,
\]
where
\begin{align*}
\Bcal_3(U) &\coloneqq -\int_\RR a' v^\b (p(v)-p(\vtil)) \rd_\x(p(v)-p(\vtil)) d\x, \\
\Bcal_4(U) &\coloneqq -\int_\RR a' p(\vtil)' (p(v)-p(\vtil))(v^\b-\vtil^\b) d\x, \\
\Bcal_5(U) &\coloneqq -\int_\RR a p(\vtil)' \rd_\x(p(v)-p(\vtil)) (v^\b-\vtil^\b) d\x.
\end{align*}
\begin{proposition}\label{prop_parabolic_out}
Under the same assumptions as Proposition \ref{prop_hyperbolic_out}, we have
\begin{align}
&\abs{\Bcal_3(U)} \le C\l\Dcal(U) + C\e\Gcal_v(U) \label{para-B3}\\
&\abs{\Bcal_4(U)} \le C \frac{\e^3}{\l} \Dcal(U) +C \e^2\Gcal_v(U) \label{para-B4}\\
&\abs{\Bcal_5(U)} \le  C\frac{\e}{\l}\Dcal(U)+C\e^2\Gcal_v(U) \label{para-B5}\\
&\abs{\Jcal^{para}(U)} \le C \Big(\l+\frac{\e}{\l}\Big)\Dcal(U) + C \e\Gcal_v(U) \label{para-total}\\
&\abs{\Jcal^{para}(U)} \le C^* \Big(\frac{\e^2}{\l} +\Big(\l+\frac{\e}{\l}\Big) \Dcal(U)\Big) \label{para-tot}
\end{align}
\end{proposition}
\begin{proof}
\bpf{para-B3} First of all, we apply H\"older's inequality to get
\begin{equation} \label{B3}
\begin{aligned}
\abs{\Bcal_3(U)} 
&\le \l \Big(\int_\RR a v^\b \abs{\rd_\x(p(v)-p(\vtil))}^2 d\x \Big)
+ \frac{1}{\l}\Big(\int_\RR |a'|^2 v^\b(p(v)-p(\vtil))^2 d\x \Big).
\end{aligned}
\end{equation}
Using Young's inequality, we get
\[
\int_\RR a' v^\b(p(v)-p(\vtil))^2 d\x \leq
\underbrace{C \int_\RR a' v^\b(p(v)-p(\vbar))^2 d\x}_{\eqqcolon \Bcal_{31}}
+ \underbrace{C\int_\RR a' v^\b(p(\vbar)-p(\vtil))^2 d\x}_{\eqqcolon \Bcal_{32}}.
\]
Before proceeding further, we make the following observation:
since \(\Phi(v/\vtil)/v\) is increasing along \(v\) on \((\vtil,+\infty)\), we have
\begin{equation} \label{Q-lin}
v\one{p(v)-p(\vtil)<-\d_3} \le C\Phif{v}{\vtil}.
\end{equation}
For $\Bcal_{31}$, we use \eqref{pwp2} with Lemma \ref{lemma_pushing}, \eqref{Qconst-sb} and \eqref{locE}, and \eqref{Q-lin} with \(\b\in[0,1]\) to have
\begin{align*}
&\abs{\Bcal_{31}(U)} = \int_\RR a' v^\b(p(v)-p(\vbar))^2\Big(\one{p(v)-p(\vtil)>\d_3} +\one{p(v)-p(\vtil)<-\d_3}\Big) d\x \\
&\leq C\Dcal(U)\int_\RR a'\int_{\x_0}^\x \one{p(v)-p(\vtil)>\d_3}d\z d\x + C\frac{\e^4}{\l^3}\Gcal_v(U) + C\int_\RR a' v\one{p(v)-p(\vtil)<-\d_3} d\x\\
&\leq C\frac{\e}{\l}\Dcal(U) + C\frac{\e^4}{\l^3}\Gcal_v(U) + C\Gcal_v(U).
\end{align*}
For \(\Bcal_{32}\), since \(\b\in[0,1]\), we use \eqref{p-quad}, \eqref{Phi-sim} and \eqref{Q-lin} to obtain
\begin{align*}
\abs{\Bcal_{32}(U)} &\leq
C\int_\RR a'\abs{p(\vbar)-p(\vtil)}^2 \one{p(v)-p(\vtil)\ge -\d_3}d\x
+C\int_\RR a' v \one{p(v)-p(\vtil)<-\d_3}d\x \leq C\Gcal_v(U).
\end{align*}
Hence, from the \(L^{\infty}\) bound of \(a'\) in \eqref{der-scale}, the following holds: 
\begin{equation} \label{para-core}
\int_\RR |a'|^2 v^\b(p(v)-p(\vtil))^2 d\x \le C\e^2 \Dcal(U) +C \e\l\Gcal_v(U).
\end{equation}
This together with \eqref{B3} implies \eqref{para-B3} as we desired.

\bpf{para-B4} It holds by \eqref{der-a} and Young's inequality that
\[
\abs{\Bcal_4(U)}
\le C \frac{\e}{\l}\int_\RR |a'|^2 v^\b(p(v)-p(\vtil))^2 d\x +C \frac{\e}{\l}\int_\RR |a'|^2 \frac{(v^\b-\vtil^\b)^2}{v^\b} d\x .
\]
For the first term on the right-hand side, we use \eqref{para-core}.\\
For the second one, we decompose it into two parts as follows:
\begin{align*}
\int_\RR |a'|^2 \frac{(v^\b-\vtil^\b)^2}{v^\b}
\one{p(v)-p(\vtil)>2\d_3} d\x
+\int_\RR |a'|^2 \frac{(v^\b-\vtil^\b)^2}{v^\b}
\one{p(v)-p(\vtil)\le2\d_3} d\x\eqqcolon \Bcal_{41} + \Bcal_{42}.
\end{align*}
For \(\Bcal_{41}\), we use \eqref{der-scale}, \eqref{powerup}, \eqref{pwp2} and Lemma \ref{lemma_pushing} with \eqref{locE} to have
\begin{align*}
\abs{\Bcal_{41}(U)} &\leq C\e\l \int_\RR a' \frac{1}{v^\b}\one{p(v)-p(\vtil)>2\d_3}
\leq  C\e\l \int_\RR a' v^\b \abs{p(v) - p(\vbar_b)}^2d\x\\
&\leq C\e\l \int_\RR a' \int_{\x_0}^\x \one{p(v)-p(\vtil)>\d_3}d\z d\x + C\frac{\e^5}{\l^2}\Gcal_v(U)
\leq C\e^2 \Dcal(U) + C\frac{\e^5}{\l^2}\Gcal_v(U).
\end{align*}
Next, we use \eqref{der-scale}, \eqref{rel_Phi}\(_1\) and \eqref{rel_Phi-L} to have
\begin{align*}
\abs{\Bcal_{42}(U)}
&\leq C \e\l \int_\RR a' |v^\b - \vtil^\b|^2\one{\abs{p(v)-p(\vtil)}\le2\d_3}
+|v^\b - \vtil^\b|\one{p(v)-p(\vtil)<-2\d_3} d\x\\
&\leq C \e\l \int_\RR a' |v - \vtil|^2\one{\abs{p(v)-p(\vtil)}\le2\d_3}
+|v - \vtil|\one{p(v)-p(\vtil)<-2\d_3} d\x
\leq C\e\l \Gcal_v(U).
\end{align*}
Therefore, gathering \(\Bcal_{41}\) and \(\Bcal_{42}\), we obtain
\begin{equation}\label{est-B4}
\int_\RR |a'|^2 \frac{(v^\b-\vtil^\b)^2}{v^\b} d\x \leq C \e^2 \Dcal(U) + C \e\l\Gcal_v(U).
\end{equation}
Thus, we have
\[
\abs{\Bcal_4(U)}
\leq C \frac{\e^3}{\l} \Dcal(U) +C \e^2\Gcal_v(U).
\]
\bpf{para-B5} We use Young's inequality with \eqref{der-a} and apply \eqref{est-B4} to get
\[
\abs{\Bcal_5(U)}
\le C\frac{\e}{\l}\Big(\int_\RR a v^\b \abs{\rd_\x(p(v)-p(\vtil))}^2 d\x
+\int_\RR |a'|^2 \frac{(v^\b-\vtil^\b)^2}{v^\b} d\x\Big)
\le C\frac{\e}{\l}\Dcal(U)+C\e^2\Gcal_v(U).
\]

\bpf{para-total}
\eqref{para-total} is a direct consequence of \eqref{para-B3}, \eqref{para-B4} and \eqref{para-B5}.

\bpf{para-tot}
\eqref{para-tot} is a direct consequence of \eqref{para-total}, \eqref{locE} and \eqref{l2}.
\end{proof}

\subsubsection{Control of shift part}\label{section-shift}
Recall the functional \(Y\) in \eqref{ybg-first} first.
We then decompose it into four parts \(Y_g,Y_b,Y_l\) and \(Y_s\) as follows:
\[
Y=Y_g+Y_b+Y_l+Y_s,
\]
where
\begin{align*}
Y_g &\coloneqq
-\frac{1}{2\s_\e^2} \int_\O a' (p(v)-p(\vtil))^2 d\x
-\int_\O a' \Phif{v}{\vtil} d\x \\
&\qquad\qquad\qquad
-\int_\O a p(\vtil)' (v-\vtil) d\x 
+\frac{1}{\s_\e} \int_\O a \htil' (p(v)-p(\vtil)) d\x, \\
Y_b &\coloneqq
-\frac{1}{2}\int_\O a'\Big(h-\htil-\frac{p(v)-p(\vtil)}{\s_\e}\Big)^2 d\x 
-\frac{1}{\s_\e}\int_\O a'(p(v)-p(\vtil))\Big(h-\htil-\frac{p(v)-p(\vtil)}{\s_\e}\Big) d\x,\\
Y_l &\coloneqq
\int_\O a \htil' \Big(h-\htil-\frac{p(v)-p(\vtil)}{\s_\e}\Big) d\x, \\
Y_s &\coloneqq
-\int_{\O^c} a'\Phif{v}{\vtil} d\x
-\int_{\O^c} a p(\vtil)' (v-\vtil) d\x
-\int_{\O^c} a'\frac{(h-\htil)^2}{2} d\x
+\int_{\O^c} a \htil' (h-\htil) d\x.
\end{align*}
Notice that \(Y_g\) consists of terms related to \(v-\vtil\), while \(Y_b\) and \(Y_l\) consist of terms related to \(h-\htil\), with \(Y_b\) being quadratic and \(Y_l\) linear in \(h - \htil\).
To apply Proposition \ref{prop:main3} in the proof of Proposition \ref{prop:main}, we need to show that \(Y_g(U)-\Ical_g^Y(\vbar),Y_b(U),Y_l(U)\), and \(Y_s(U)\) are negligible by the good terms, which is stated in the following proposition.
\begin{proposition} \label{prop:shift}
For any constant \(C^*>0\), there exist constants \(\e_0,\d_0,C>0\) (in particular, \(C\) depends on \(C^*\)) such that for any \(\e<\e_0\) and \(\d_0^{-1}\e <\l<\d_0<1/2\), the following statements hold.
For any \(U\) such that \(\abs{Y(U)}\le\e^2\) and \(\Dcal(U)\le C^*\frac{\e^2}{\l}\),
\begin{equation}\label{Y-conc}
\begin{aligned}
&\abs{Y_g(U)-\Ical_g^Y(\vbar)}^2+\abs{Y_b(U)}^2+\abs{Y_l(U)}^2+\abs{Y_s(U)}^2 \\
&\le C\frac{\e^2}{\l}\bigg(\frac{\e}{\l}\Dcal(U)+\sqrt{\frac{\e}{\l}}\Gcal_v(U)+(\Gcal_v(U)-\Gcal_v(\Ubar))
+\Gcal_h^-(U) +\sqrt{\frac{\l}{\e}}\Gcal_h^+(U)\bigg)
\end{aligned}
\end{equation}
\end{proposition}
\begin{proof}
We split the proof into three steps.
\step{1}
First of all, we use the notations \(Y_1^s,Y_2^s,Y_3^s\) and \(Y_4^s\) for the terms of \(Y_s\) as follows:
\[
Y_s =
\underbrace{-\int_{\O^c} a'\Phif{v}{\vtil} d\x}_{\eqqcolon Y_1^s}
\underbrace{-\int_{\O^c} a p(\vtil)' (v-\vtil) d\x}_{\eqqcolon Y_2^s}
\underbrace{-\int_{\O^c} a'\frac{(h-\htil)^2}{2} d\x}_{\eqqcolon Y_3^s}
\underbrace{-\int_{\O^c} a \htil' (h-\htil) d\x}_{\eqqcolon Y_4^s}.
\]
Then, we observe the following:
\begin{align*}
&\abs{Y_g(U)-\Ical_g^Y(\vbar)}+\abs{Y_1^s(U)-Y_1^s(\Ubar)}+\abs{Y_1^s(\Ubar)}
+\abs{Y_2^s(U)-Y_2^s(\Ubar)}+\abs{Y_2^s(\Ubar)} \\
&\qquad\le
C\int_\O a' \abs{p(v)-p(\vbar)}^2 d\x
+C\int_\O a' \abs{p(v)-p(\vbar)} d\x
+C\int_{\O^c} a' d\x\\
&\qquad\qquad
+C\int_\RR a' \Big(\Phif{v}{\vtil}-\Phif{\vbar}{\vtil}\Big) d\x
+C\int_\RR a' \abs{v-\vbar} d\x.
\end{align*}
Note that the first line of the right-hand side can be controlled exactly the same as \eqref{bo1p}.
Moreover, we use \eqref{rel_Phi-L} to bound the remaining terms.
Hence, we obtain
\begin{equation} \label{YgYs12}
\begin{aligned}
&\abs{Y_g(U)-\Ical_g^Y(\vbar)}+\abs{Y_1^s(U)-Y_1^s(\Ubar)}+\abs{Y_2^s(U)-Y_2^s(\Ubar)}+ \abs{Y_l(U)}+\abs{Y_s(U)}  \\
&\qquad\le C\frac{\e}{\l}\Dcal(U)+C\frac{\e^6}{\l^4}\Gcal_v(U)+C(\Gcal_v(U)-\Gcal_v(\Ubar)).
\end{aligned}
\end{equation}
On the other hand, by the definitions of \(\Gcal_h^+, \Gcal_h^-\) and \(\Bcal_1^+\), we have 
\[
\abs{Y_3^s(U)}+\abs{Y_b(U)}
\le C\big(\Gcal_h^+(U)+\Gcal_h^-(U)+\abs{\Bcal_1^+(U)}\big).
\]
Moreover, since 
\[
\Gcal_h^+(U)
\le C\int_\O a' \Big((h-\htil)^2+(p(v)-p(\vtil))^2\Big)d\x
\le C\int_\O a' (h-\htil)^2d\x + C\abs{\Bcal_1^+(U)},
\]
it holds from \eqref{bo} that 
\[
\abs{Y_3^s(U)}+\abs{Y_b(U)} \le C\int_\RR a' (h-\htil)^2d\x
+C^*\Big(\frac{\e^2}{\l} + \Big(\frac{\e}{\l}\Big)^{\frac{3}{4}} \Dcal(U)\Big).
\]
Therefore, using \eqref{locE}, \eqref{l2} and the assumption \(\Dcal(U)\le C^*\frac{\e^2}{\l}\), it follows from \eqref{YgYs12} and the above estimate that 
\[
\abs{Y_g(U)-\Ical_g^Y(\vbar)}+\abs{Y_b(U)}+\abs{Y_1^s(U)}+\abs{Y_2^s(U)}+\abs{Y_3^s(U)} \le C\frac{\e^2}{\l}.
\]
\step{2}
First of all, using Young's inequality, \eqref{p-quad}, \eqref{Phi-sim} and \eqref{bo1p}, we estimate
\begin{align*}
&\abs{\int_\O a'(p(v)-p(\vtil))\Big(h-\htil-\frac{p(v)-p(\vtil)}{\s_\e}\Big)d\x}
\le \sqrt{\frac{\l}{\e}}\Gcal_h^+(U)
+C\sqrt{\frac{\e}{\l}} \Bcal_1^+(U) \\
&\qquad\qquad\le \sqrt{\frac{\l}{\e}}\Gcal_h^+(U)
+C\sqrt{\frac{\e}{\l}} \Big(\Ical_1(\vbar)+(\Bcal_1^+(U)-\Ical_1(\vbar))\Big) \\
&\qquad\qquad\le \sqrt{\frac{\l}{\e}}\Gcal_h^+(U)
+C\sqrt{\frac{\e}{\l}} \Big(C\Gcal_v(U)+C\frac{\e}{\l}\Dcal(U)\Big).
\end{align*}
Therefore, this together with \eqref{YgYs12} yields 
\begin{align*}
&\abs{Y_g(U)-\Ical_g^Y(\vbar)}+\abs{Y_b(U)}+\abs{Y_1^s(U)}+\abs{Y_2^s(U)}+\abs{Y_3^s(U)} \\
&\le C\frac{\e}{\l}\Dcal(U)
+C\sqrt{\frac{\e}{\l}}\Gcal_v(U)
+C(\Gcal_v(U)-\Gcal_v(\Ubar))
+C\Gcal_h^-(U)
+C\sqrt{\frac{\l}{\e}}\Gcal_h^+(U).
\end{align*}

\step{3}
Using H\"older's inequality and \eqref{der-scale} yields
\begin{align*}
\abs{Y_l(U)}^2+\abs{Y_4^s(U)}^2
&\le \bigg(\int_\RR a\htil' d\x\bigg) \bigg(\int_\O a\htil'  \Big(h-\htil -\frac{p(v)-p(\vtil)}{\s_\e}\Big)^2  d\x + \int_{\O^c} a\htil' (h-\htil)^2 d\x \bigg) \\
&\le C\frac{\e^2}{\l} (\Gcal_h^+(U)+\Gcal_h^-(U)).
\end{align*}
Therefore, this together with \(Step \,1\) and \(Step \,2\) yields \eqref{Y-conc} as we desired.
\end{proof}

\subsection{Proof of Proposition \ref{prop:main}}\label{section-main}
We now prove the main Proposition \ref{prop:main}.
We split the proof into two steps, depending on the strength of the dissipation term \(\Dcal(U)\).
\step{1}
We first consider the case of \(\Dcal(U)\ge 10C^*\frac{\e^2}{\l}\), where the constant \(C^*\) is chosen to be the greater one between constants \(C^*\) given in Proposition \ref{prop_hyperbolic_out} and Proposition \ref{prop_parabolic_out}.
Then, using \eqref{bo} and \eqref{para-tot} and taking \(\d_0\) small enough, we have
\begin{align*}
\Rcal(U)
&\le
2\abs{\Bcal_{\d_3}(U)} + \abs{\Bcal_1^+(U)} + 2\abs{\Jcal^{para}(U)} - (1-\d_0)\Dcal(U) \\
&\le 5C^*\frac{\e^2}{\l} - (1-C\d_0-C\sqrt{\d_0})\Dcal(U)
\le 5C^*\frac{\e^2}{\l} - \frac{1}{2}\Dcal(U) \le 0,
\end{align*}
which gives the desired result. 

\step{2}
We now assume the other alternative, i.e., \(\Dcal(U)\le 10C^* \frac{\e^2}{\l}\). 
Since we have \eqref{lbis}, we choose the small constant \(\d_3>0\) of Proposition \ref{prop:main3} associated to the constant \(C_2\) in \eqref{lbis}.
Then we define the truncation of size \(\d_3\).
Using 
\[
\Ical_g^Y(\vbar) = Y(U) - (Y_g(U)-\Ical_g^Y(\vbar)) - Y_b(U) - Y_l(U) - Y_s(U),
\]
we have 
\[
-5\abs{Y(U)}^2 \le -\abs{\Ical_g^Y(\vbar)}^2 + 5\abs{Y_g(U)-\Ical_g^Y(\vbar)}^2 + 5\abs{Y_b(U)}^2 + 5\abs{Y_l(U)}^2 + 5\abs{Y_s(U)}^2.
\]
Now, taking \(\d_0>0\) sufficiently small such that \(\d_0 \le \d_3^5\), we then use \eqref{keyD} to obtain that for \(\e<\e_0(\le \d_3)\) and \(\e/\l<\d_0\), 
\begin{align*}
\Rcal(U) &\le -\frac{5}{\e\d_3}Y(U)^2 + \Bcal_{\d_3}(U) + \d_0\frac{\e}{\l}\abs{\Bcal_{\d_3}(U)} + \d_0\frac{\e}{\l}\Bcal_1^+(U) + \Jcal^{para}(U) + \d_0\abs{\Jcal^{para}(U)} \\
&\qquad -\Gcal_h^-(U)-\frac{1}{2}\Gcal_h^+(U) -\left(1-\d_0\frac{\e}{\l}\right)\Gcal_v(U) - (1-\d_0)\Dcal(U) \\
&\le -\frac{1}{\e\d_3}\abs{\Ical_g^Y(\vbar)}^2 + \Ical_1(\vbar) + \Bcal_2(\Ubar) 
+ \d_0\frac{\e}{\l}\left(\abs{\Ical_1(\vbar)} + \abs{\Bcal_2(\Ubar)}\right) -\left(1-\d_3\frac{\e}{\l}\right)\Gcal_v(\Ubar) \\
&\qquad - (1-\d_3)\Dcal(\Ubar) 
+\underbrace{
\left(1+\d_0\frac{\e}{\l}\right)\big(
\abs{\Bcal_1^+(U)-\Ical_1(\vbar)} + \abs{\Bcal_1^-(U)} + \abs{\Bcal_2(U)-\Bcal_2(\Ubar)}
\big)
}_{\eqqcolon J_1} \\
&\qquad
+\underbrace{\left(1+\d_0\right)\abs{\Jcal^{para}(U)}}_{\eqqcolon J_2}
+\underbrace{\frac{5}{\e\d_3}\left(\abs{Y_g(U)-\Ical_g^Y(\vbar)}^2 + \abs{Y_b(U)}^2 + \abs{Y_l(U)}^2 + \abs{Y_s(U)}^2\right)}_{\eqqcolon J_3} \\
&\qquad
-\Gcal_h^-(U)-\frac{1}{2}\Gcal_h^+(U)-\frac{1}{2}(\Gcal_v(U)-\Gcal_v(\Ubar))
- (\d_3-\d_0)\frac{\e}{\l}\Gcal_v(\Ubar) 
- (\d_3-\d_0)\Dcal(U).
\end{align*}
We now claim that \(J_1, J_2, J_3\) are controlled by the remaining good terms above. 
Indeed, it follows from \eqref{bo1p}-\eqref{bo2} that for any \(\l, \e/\l<\d_0(\le \d_3^5)\),
\begin{align*}
J_1 &\le \d_3 \Gcal_h^-(U) + C\Big(\frac{\e}{\l}\Big)^{\frac{3}{4}}\Dcal(U)+ C\frac{\e}{\l}(\Gcal_v(U)-\Gcal_v(\Ubar))+ C\frac{\e^6}{\l^4} \Gcal_v(U)  \\
&\le \d_3\Gcal_h^-(U) + \frac{\d_3}{4}\Dcal(U)
+ \frac{\d_3}{4}(\Gcal_v(U)-\Gcal_v(\Ubar))
+ \frac{\d_3}{4}\frac{\e}{\l}\Gcal_v(\Ubar)
\end{align*}
and from \eqref{para-total},
\begin{align*}
J_2 &\le C \Big(\l+\frac{\e}{\l}\Big)\Dcal(U) + C\e\Gcal_v(U)
\le \frac{\d_3}{4} \Dcal(U)
+ \frac{\d_3}{4}(\Gcal_v(U)-\Gcal_v(\Ubar))
+ \frac{\d_3}{4}\frac{\e}{\l}\Gcal_v(\Ubar).
\end{align*}
For \(J_3\), from \eqref{Y-conc}, we have
\begin{align*}
J_3 
&\le \frac{C}{\d_3}\sqrt{\frac{\e}{\l}}
\Big(\Dcal(U)
+ \Gcal_h^-(U) + \Gcal_h^+(U)
+(\Gcal_v(U)-\Gcal_v(\Ubar))
+\frac{\e}{\l}\Gcal_v(\Ubar)\Big) \\
&\le \frac{\d_3}{4}\Big(\Dcal(U) + \Gcal_h^-(U)+\Gcal_h^+(U)
+(\Gcal_v(U)-\Gcal_v(\Ubar)) 
+\frac{\e}{\l}\Gcal_v(\Ubar)\Big).
\end{align*}
Therefore, we have 
\begin{align*}
\Rcal(U) &\le -\frac{1}{\e\d_3}\abs{\Ical_g^Y(\vbar)}^2 + \Ical_1(\vbar) + \Bcal_2(\Ubar) 
+ \d_0\frac{\e}{\l}\left(\abs{\Ical_1(\vbar)} + \abs{\Bcal_2(\Ubar)}\right) \\
&\qquad-\left(1-\d_3\frac{\e}{\l}\right)\Gcal_v(\Ubar) - (1-\d_3)\Dcal(\Ubar).
\end{align*}
Since the above quantities \(\Ical_g^Y(\vbar)\), \(\Bcal_2(\Ubar)\), \(\Gcal_v(\Ubar)\), and \(\Dcal(\Ubar)\) depends only on \(\vbar\), and 
\[
\Bcal_2(\Ubar) = \Ical_2(\vbar), \qquad
\Dcal(\Ubar) = \Dcal(\vbar),
\]
it holds by Proposition \ref{prop:main3} that \(\Rcal(U)\le 0\). 
This completes the proof of Proposition \ref{prop:main}. \qed


\section{Proof of Theorem \ref{thm_inviscid}: Inviscid Limits and Stability}
\setcounter{equation}{0}

The proof of Theorem \ref{thm_inviscid} is largely similar to that in \cite{KV-Inven}, whose key ingredient is on the uniform-in-\(\n\) estimate obtained from Theorem \ref{thm_main}.
In particular, for the existence of well-prepared initial data \eqref{ini_conv}, the existence of inviscid limits \eqref{wconv} and the stability estimate \eqref{uni-est}, the proofs are essentially the same.
So, we present the proof of \eqref{X-control} only.

Firstly, we obtain the uniform-in-\(\n\) estimates from the contraction estimate of Theorem \ref{thm_main}.
Let \(\{ (v^\nu,u^\nu) \}_{\nu>0}\) denote a sequence of solutions to \eqref{inveq}-\eqref{pressure} on \((0,T)\) subject to the sequence of initial datum \(\{(v_0^\nu,u_0^\nu)\}_{\n>0}\).
Moreover, throughout this section, let \(C\) denote a positive constant that may vary from line to line but remains independent of \(\n\). (This may depend on \(\e,\l\) and \(T\) now.)

We apply Theorem \ref{thm_main} to the rescaled functions
\[
(v,u)(t,x) = (v^\n, u^\n)(\n t, \n x), \quad \text{and} \quad (\vtil,\util)(x) = (\vtil^\n, \util^\n)(\n x),
\]
perform the change of variables \(t \mapsto t/\nu \) and \(x \mapsto x/\nu \) and make use of \eqref{ini_conv} so that we obtain that for any \(\d\in(0,1)\), there exists a constant \(\nu_*>0\) such that for any \(\nu<\nu_*\) and any \(t\in[0,T]\), the following holds:
\begin{align}
\begin{aligned} \label{ineq-m}
&\int_\RR \eta\left((\vn,\hn)(t,x)|(\vtn,\htn)(x-X_\nu(t))\right)dx \\
&\qquad +\int_0^{T} \int_\RR \abs{(\vtn)'(x-X_\nu(s))} \Phi(\vn(s,x)/\vtn(x-X_\nu(s))) dxds \\
&\qquad +\nu\int_0^{T} \int_\RR (\vn)^\b(s,x)\abs{\rd_x\big(p(\vn(s,x))-p(\vtn(x-X_\nu(s)))\big)}^2 dxds 
\le C \Ecal_0 + \d,
\end{aligned}
\end{align}
where \(\Ecal_0 \coloneqq \int_\RR \eta \left( (v^0,u^0) | (\vbar,\ubar) \right) dx\), \(X_\nu(t)\coloneqq \nu X(t/\nu)\), and
\[
h^\n \coloneqq u^\n - \n \frac{(\vn)_x}{(\vn)^{\a+1}}, \qquad \htil^\n \coloneqq \util^\n - \n \frac{(\vtn)_x}{(\vtn)^{\a+1}}.
\]

To show \eqref{X-control}, it is also necessary to demonstrate the convergence of \(\{X_\nu\}_{\nu>0}\), as stated in the following lemma.
\begin{lemma}\label{lem-X}
There exists \(X_\infty \in BV((0,T))\) such that
\begin{align*}
X_\nu \to X_\infty \quad \text{in } L^1(0,T), \quad \text{up to subsequence as } \nu \to 0.
\end{align*}
\end{lemma}

\begin{proof}
Recall that \(X_\nu(t)= \nu X(t/\nu)\) and \(X'_\nu(t)= X'(t/\nu)\).
Then, thanks to \eqref{est-shift}, we get
\[
\abs{X'_\nu(t)} \le C \left( f_\nu(t) + 1 \right)    
\]
where \(f_\nu(t)\coloneqq f(t/\nu)\). Hence, it follows from \eqref{est-shift} and \eqref{ini_conv} that for any \(\nu<\nu_*\),
\[
\norm{f_\nu}_{L^1(0,T)}
=\nu \norm{f}_{L^1(0,T/\nu)}
\le C\nu \int_\RR E\big((v_0^\n, u_0^\n)(\n x)|(\vtn, \utn)(\n x)\big) dx
\le C \left( \Ecal_0 + 1 \right).
\]
This implies that \(\{f_\nu\}_{\nu>0}\) is bounded in \(L^1(0,T)\) and so \(X_\nu'\) is uniformly bounded in \(L^1(0,T)\).
Moreover, from \(X_\nu(0)=0\), we have
\begin{equation}\label{X-unif}
\abs{X_\nu(t)} \le Ct+C\int_0^t f_\nu(s) ds.
\end{equation}
In other words, \(X_\nu\) is uniformly bounded in \(L^1(0,T)\) as well.
Therefore, the compactness of BV (see for example \cite[Theorem 3.23]{AmbroFuscoPall00}) gives the desired result.
\end{proof}
We now present the proof of \eqref{X-control}.
The proof is motivated by \cite[Section 6.2]{EEK-BNSF}.
Thanks to \eqref{X-unif}, we choose a positive constant \(r=r(T)>1\) such that \(\norm{X_\n}_{L^\infty(0,T)}\le r/3\) for any \(\n\in(0,\n_*)\).
Note that \(r\) is a constant depends only on \(\Ecal_0\), not on \(\n\); we may take \(r\) to satisfy \(r\le C(\Ecal_0+T+1)\).
Then, we also consider a nonnegative smooth function \(\ps\colon \RR\to \RR\) such that \(\ps(x) = \ps(-x)\), \(\ps'(x)\le 0\) for all \(x\ge 0\), \(\abs{\ps'(x)}\le 2/r\) for all \(x\in \RR\), and
\[
\ps(x) = \begin{cases}
    1, &\text{ if } \abs{x}\le r, \\
    0, &\text{ if } \abs{x}\ge 2r.
\end{cases}
\]
Let \(\th\colon \RR\to \RR\) be a nonnegative smooth function such that \(\th(s) = \th(-s)\), \(\int_\RR \th = 1\) and \(\supp \th \subseteq [-1, 1]\), and for any \(\e>0\), let 
\[
\th_\e(s)\coloneqq \frac{1}{\e}\th(\frac{s-\e}{\e}). 
\]
Here, \(\e>0\) is a parameter, not the shock strength.
Then, for any given \(t\in(0, T)\) and any \(\e<t/2\), we define a nonnegative smooth function 
\[
\phi_{t, \e}(s) \coloneqq \int_0^s\left(\th_\e(\t)-\th_\e(\t-t)\right)d\t. 
\]
Since \(v_t^\nu-h_x^\nu = \n(\frac{v_x^\nu}{(\vn)^{\a+1}})_x\), we simply get
\begin{align*}
&\int_{[0, T]\times \RR} \left(\phi_{t, \e}'(s)\ps(x)\vn(s, x) - \phi_{t, \e}(s)\ps'(x)\hn(s,x) dsdx\right)dsdx\\
&\qquad=\int_{[0,T]\times\RR} \phi_{t,\e}(s)\ps'(x)\n\Big(\frac{v_x^\nu}{(\vn)^{\a+1}}\Big)(s,x)dsdx.
\end{align*}
Then, since the solutions belong to the function space \(\Hcal_T\), it holds that \(v_t^\nu\in L^2(0,T;L^2(\RR))\).
Hence, the continuity of \(\int_\RR \ps(x)\vn(s,x)dx\) with respect to \(s\) variable follows:
\begin{align*}
\abs{\int_\RR \ps(x)\vn(s,x)dx - \int_\RR \ps(x)\vn(s',x)dx}
&\le \int_{\abs{x}\le 2r}\int_s^{s'} \abs{v_t^\nu(t,x)} dtdx\\
&\le \sqrt{4r\abs{s-s'}} \norm{v_t^\nu}_{L^2(0,T;L^2(\RR))}.
\end{align*}
Notice that \(\vn,\hn\) and \((\frac{v_x^\nu}{(\vn)^{\a+1}})\) are all locally integrable.
Hence, we apply the dominated convergence theorem as \(\e\to0\) to obtain:
\begin{equation} \label{decomp}
\begin{aligned}
&\int_\RR \ps(x)v_0^\nu(x)dx - \int_\RR \ps(x)\vn(t,x)dx
-\int_0^t\int_\RR \ps'(x)\hn(s,x)dxds\\
&\qquad=\n\int_0^t\int_\RR \ps'(x)\Big(\frac{v_x^\nu}{(\vn)^{\a+1}}\Big)(s,x)dxds \eqqcolon J_1.
\end{aligned}
\end{equation}
Now, we decompose the left-hand side of \eqref{decomp} to find that
\begin{equation} \label{decomp-1}
\begin{aligned}
0&=J_1
+\underbrace{\int_\RR \ps(x) \Big(\vn(t,x)-\vtn(x-X_\n(t))\Big) dx}_{\eqqcolon J_2}
+\underbrace{\int_\RR \ps(x) (\vtn(x)-v_0^\nu(x)) dx}_{\eqqcolon J_3}\\
&\quad +\underbrace{\int_0^t \int_\RR \ps'(x)\Big(\hn(s,x)-\htn(x-X_\n(t))\Big)dxds}_{\eqqcolon J_4}\\
&\quad +\underbrace{\int_\RR \ps(x)\Big(\vtn(x-X_\n(t))-\vtn(x)\Big)dxds
+\int_0^t\int_\RR \ps'(x)\htn(x-X_\nu(s))dxds}_{\eqqcolon J_5}.
\end{aligned}
\end{equation}
The terms \(J_2,J_3,J_4,\) and \(J_5\) can be estimated in essentially the same manner as in \cite[Section 6.2]{EEK-BNSF}.
Hence, we do not present the details here; the bound is given as follows: 
\[
\abs{J_2}+\abs{J_3}+\abs{J_4}+\abs{J_5-(X_\n(t)-\s t)(v_--v_+)}
\le C(T)(\sqrt{\Ecal_0+\d}+\Ecal_0+\d)+C(\Ecal_0,T)\n.
\]
It remains to control \(J_1\).
Since we have \eqref{tail} and
\[
\abs{x-X_\n(s)} \ge \abs{x} - r/3 \ge r/2 \quad \text{for all } x\in[-2r,-r]\cup[r,2r],
\]
it holds that for any \(\n<\n_*\), 
\[
\sup_{r\le\abs{x}\le2r} (\vtn)'(x-X_\n(s)) \le \sup_{r\le\abs{x}\le2r} \frac{1}{\n}\vt'(\frac{x-X_\n(s)}{\n}) \le \frac{C}{\n}e^{-C\frac{r}{2\n}} \le C.
\]
Then, considering the support of \(\ps\) and the upper bound of \(\abs{\ps'}\), we get
\begin{align*}
&\abs{J_1}
\le C\n \hspace{-1mm}\int_0^t \hspace{-1mm}\int_{r\le\abs{x}\le2r} \abs{\frac{v_x^\nu}{(\vn)^{\a+1}}(s,x)}dxds\\
&\le C\n \hspace{-1mm}\int_0^t \hspace{-1mm}\int_{r\le\abs{x}\le2r} \hspace{-1mm}(\vn)^\b(s,x) \big(\abs{(p(\vn)(s,x)-p(\vtn)(x-X_\n(s)))_x}\hspace{-0.5mm}+\hspace{-0.5mm}\abs{p(\vtn)'(x-X_\n(s))}\big)dxds\\
&\le C\n \hspace{-1mm}\int_0^t \hspace{-1mm}\int_{r\le\abs{x}\le2r} (\vn)^\b(s,x) \abs{(p(\vn)(s,x)-p(\vtn)(x-X_\n(s)))_x}^2 dxds\\
&\qquad+C\n \hspace{-1mm}\int_0^t \hspace{-1mm}\int_{r\le\abs{x}\le2r} (\vn)^\b(s,x) dxds.
\end{align*}
The last term can be controlled with the inequality \(v^\b \le 1+\Phi(v/\vtil)\), and thus, using \eqref{ineq-m},
\[
\abs{J_1}\le C(\Ecal_0+\d)+C(\Ecal_0,T)\n.
\]
Thus, gathering all with \eqref{decomp-1}, it holds that for any \(\d>0\), there exists \(\n_*>0\) such that \(\n<\n_*\) implies
\[
\abs{X_\n(t)-\s t} \abs{v_+-v_-} \le C(T)(\sqrt{\Ecal_0+\d}+\Ecal_0+\d)+C(\Ecal_0,T)\n.
\]
Note that the \(L^1\) convergence in Lemma \ref{lem-X} gives pointwise convergence up to subsequence.
Therefore, for a.e. \(t\in(0,T)\), being a pointwise limit of \(X_\n(t)\), \(X_\infty(t)\) satisfies 
\[
\abs{X_\infty(t)-\s t} \abs{v_+-v_-} \le C(T)(\sqrt{\Ecal_0}+\Ecal_0).
\]
This completes the proof of \eqref{X-control}. \qed


\bibliographystyle{plain}
\bibliography{reference} 

\end{document}